\documentclass[a4paper, 12pt, reqno, dvipsnames]{amsart}

\usepackage[normalem]{ulem} 
\usepackage{lmodern}
\usepackage[english]{babel}
\usepackage[latin1]{inputenc}
\usepackage{microtype}
\usepackage{amsfonts,amssymb,amsmath,amsthm,mathrsfs}
\usepackage{mathtools}
\usepackage{color, graphicx}
\usepackage[all]{xy}
\usepackage{hyperref} 
\usepackage[inline]{enumitem}
\usepackage{cleveref}
\usepackage{a4wide}
\usepackage{upgreek}
\usepackage{etoolbox}

\makeatletter
\patchcmd{\@setaddresses}{\indent}{\noindent}{}{}
\patchcmd{\@setaddresses}{\indent}{\noindent}{}{}
\patchcmd{\@setaddresses}{\indent}{\noindent}{}{}
\patchcmd{\@setaddresses}{\indent}{\noindent}{}{}
\makeatother

\makeatletter
\DeclareMathSizes{\@xpt}{\@xpt}{6}{5}
\makeatother

\makeatletter
\def\namedlabel#1#2{\begingroup
    #2%
    \def\@currentlabel{#2}%
    \phantomsection\label{#1}\endgroup
}
\makeatother

\allowdisplaybreaks
\theoremstyle{plain}

\newtheorem{theorem}{Theorem}[section]
\newtheorem{lemma}[theorem]{Lemma}
\newtheorem{proposition}[theorem]{Proposition}
\newtheorem{corollary}[theorem]{Corollary}
\newtheorem*{theorem*}{Theorem}

\theoremstyle{definition}
\newtheorem{definition}[theorem]{Definition}
\newtheorem{example}[theorem]{Example}

\theoremstyle{remark}
\newtheorem{remark}[theorem]{Remark}

\def\ol{\overline}

\newcommand{\N}{\mathbb{N}}
\newcommand{\Z}{\mathbb{Z}}
\newcommand{\Q}{\mathbb{Q}}

\newcommand{\C}{\mathbb{C}}
\newcommand{\K}{\Bbbk}

\newcommand{\cl}[1]{\overline{#1}} 
\newcommand{\coinv}[2]{{#1}{}^{\mathrm{co}{#2}}} 

\newcommand{\op}[1]{#1^{\mathrm{op}}} 
\newcommand{\mf}[1]{\mathfrak{#1}} 
\newcommand{\ms}[1]{\mathsf{#1}}

\renewcommand{\ker}{\mathsf{ker}} 

\newcommand{\id}{\mathsf{Id}} 

\newcommand{\Hom}[6]{{_{#1}^{#2}\mathsf{Hom}_{#3}^{#4}}\left({#5},{#6}\right)} 
\newcommand{\End}[2]{\mathsf{End}_{#1}(#2)} 

\newcommand{\alg}{\mathsf{Alg}}
\newcommand{\algk}{\alg_{\K}} 
\newcommand{\vectk}{\mathsf{Vect}_{\K}} 
\newcommand{\Set}{{\mathsf{Set}}} 
\newcommand{\Top}{{\mathsf{Top}}} 
\newcommand{\Ab}{{\mathsf{Ab}}} 
\newcommand{\Bim}[2]{{}_{#1}\M{}_{#2}} 
\newcommand{\Rmod}[1]{\M_{#1}} 
\newcommand{\Bimod}[1]{\Bim{#1}{#1}} 
\newcommand{\gparcom}[1]{\mathsf{gPCom}^{#1}} 
\newcommand{\com}[1]{\mathsf{Com}^{#1}} 
\newcommand{\ncom}[1]{\mathsf{NCom}^{#1}} 
\newcommand{\pcd}[1]{\mathsf{PCD}^{#1}} 
\newcommand{\ParRep}{\mathsf{PRep}} 
\newcommand{\ParMod}{\mathsf{PMod}} 
\newcommand{\Cov}{\mathsf{Cov}} 
\newcommand{\Com}{\mathsf{Com}}
\newcommand{\Mod}{\mathsf{Mod}}
\newcommand{\gphopfmod}[1]{\mathsf{gHPCom}_{#1}^{#1}}

\newcommand{\Ind}{\mathsf{Ind}}

\newcommand{\cB}{{\mathcal B}}
\newcommand{\cC}{{\mathcal C}}
\newcommand{\cD}{{\mathcal D}}

\newcommand{\cF}{{\mathcal F}}
\newcommand{\cG}{{\mathcal G}}

\newcommand{\cI}{{\mathcal I}}
\newcommand{\cJ}{{\mathcal J}}

\newcommand{\cP}{{\mathcal P}}

\newcommand{\cZ}{{\mathcal Z}}

\newcommand{\M}{\mathsf{Mod}} 
\newcommand{\I}{\mathbb{I}} 
\newcommand{\clambda}{\mf{l}}
\newcommand{\crho}{\mf{r}}

\newcommand{\ot}{\otimes}
\newcommand{\bul}{\bullet}

\newbox\pullbackbox
\setbox\pullbackbox=\hbox{\xy 0;<1mm,0mm>: \POS(2,0)\ar@{-} (0,2) \ar@{-} (4,2)
\endxy}
\def\pullback{\copy\pullbackbox}

\newbox\pushoutbox
\setbox\pushoutbox=\hbox{\xy 0;<1mm,0mm>: \POS(2,2)\ar@{-} (0,0) \ar@{-} (4,0)
\endxy}
\def\pushout{\copy\pushoutbox}

\title[Globalization of geometric partial comodules in Abelian categories]{Geometric partial comodules over flat coalgebras in Abelian categories are globalizable}

\author{Paolo Saracco}
\address{D\'epartement de Math\'ematique, Universit\'e Libre de Bruxelles, Boulevard du Triomphe, B-1050 Brussels, Belgium.}
\urladdr{\url{sites.google.com/view/paolo-saracco}}
\urladdr{\url{homepages.ulb.ac.be/~psaracco}}
\email{paolo.saracco@ulb.be}

\author{Joost Vercruysse}
\address{D\'epartement de Math\'ematique, Universit\'e Libre de Bruxelles, Boulevard du Triomphe, B-1050 Brussels, Belgium.}
\urladdr{\url{homepages.ulb.ac.be/~jvercruy}}
\email{jvercruy@ulb.be}

\thanks{PS is a Charg\'e de Recherches of the Fonds de la Recherche Scientifique - FNRS and a member of the National Group for Algebraic and Geometric Structures and their Applications (GNSAGA-INdAM). 
JV thanks the FNRS (National Research Fund of the French speaking community in Belgium) for support via the MIS project `Antipode' (Grant F.4502.18).}

\begin{document}

\begin{abstract}
The aim of this paper is to prove the statement in the title. As a by-product, we obtain new globalization results in cases never considered before, such as partial corepresentations of Hopf algebras. Moreover, we show that for partial representations of groups and Hopf algebras, our globalization coincides with those described earlier in literature. Finally, we introduce Hopf partial comodules over a bialgebra as geometric partial comodules in the monoidal category of (global) modules. By applying our globalization theorem we obtain an analogue of the fundamental theorem for Hopf modules in this partial setting.
\end{abstract}

\maketitle



\section{Introduction}

Originally introduced as a method to classify certain classes of $C^*$-algebras as generalized crossed products \cite{Exel1}, partial group actions quickly attracted the attention of the algebra community and since the beginning of the century, they have been intensively studied from a purely algebraic point of view, resulting in
remarkable applications and theoretic developments. 
Currently, partial actions and related structures are studied at various levels of generality, including partial (co)actions of (weak or multiplier) Hopf algebras \cite{Caenepeel-Janssen,Castro,Fonseca}, semigroups (e.g., recently, \cite{Khrypchenko}), inductive constellations \cite{Gould}, groupoids, and, more generally, categories \cite{Nystedt}.
For a more exhaustive summary and an idea of the impact of partial (co)actions on contemporary Mathematics, we refer the reader to the recent survey \cite{Doku2} and, in particular, to the references therein.

One of the main sources of examples of partial actions of groups is provided by the \emph{restriction} of the global action of a group $G$ on a set $X$ to arbitrary subsets of $X$. In fact, it has been proved that any partial group action can be realized as the restriction of a certain (minimal) global action, called its \emph{globalization} or \emph{enveloping action}. Globalizations proved immediately to play an important role in the study of partial actions, for instance in the development of Galois theory of partial group actions in \cite{DokuFerreroPaques} as well as in a series of other ring theoretic and Galois theoretic investigations.
Thus, a relevant question which arose very soon is the problem of existence and uniqueness of such globalizations.

A considerable number of globalization results already exists in the literature, either in a complete form (that is, existence and uniqueness) or in the form of sufficient conditions and criteria for a partial action to admit a globalization. For instance, we have globalization theorems for partial actions of groups on topological spaces (\cite{Abadie,KellendorkLawason,Novikov}), on unital associative algebras (\cite{DokuExel}), on $s$-unital rings (\cite{DokuDelRioSimon}), on semiprime rings (\cite{BemmFerrero,CortesFerrero}), we have globalization theorems for partial actions of Hopf algebras on unital algebras \cite{AlvesBatista2,AlvesBatista} (which however are not necessarily unital), on $\K$-linear categories \cite{AlvaresBatista}, for twisted partial actions of Hopf algebras \cite{AlvesBatistaDokuPaques}, partial modules over Hopf algebras \cite{AlvesBatistaVercruysse}, partial actions of multiplier Hopf algebras \cite{Fonseca}, and we have globalization theorems for partial groupoid actions on rings \cite{BagioPaques}, on $s$-unital rings \cite{BagioPinedo}, on $R$-categories \cite{MarinPinedo}. Again, we refer to \cite{Doku2} for further references and for a glimpse of the importance of the globalization procedure. 

Lead by the need of a unified approach to the vast panorama of the theory of partial actions and in order to tackle the globalization problem, in \cite{Saracco-Vercruysse} we discussed the question from a category-theoretical perspective by taking advantage of the notion of \emph{geometric partial comodule} introduced in \cite{JoostJiawey}. We also provided a genuine procedure to construct globalizations (whenever they exist) that can be applied to many concrete cases of interest. 
In \cite{Saracco-Vercruysse1}, we already studied globalization of geometric partial comodules in the opposite of the categories of sets and topological spaces and in the category of algebras over a commutative ring (or, more precisely, we studied the globalization of geometric partial modules in $\Set$, $\Top$ and $\op{\algk}$). In the present paper we analyse the globalization problem in abelian categories, proving that in this setting globalization always exists under very mild conditions, and we apply this result in several concrete situations, recovering on the one hand some known globalization constructions, and providing on the other hand original ones, which have been never studied before.

We believe that the globalization results obtained in this paper should help to understand the connection between geometric partial comodules and other structures studied in literature, such as {\em glider representations} \cite{CaeVO}, which were recently proved to have a rich homological structure \cite{RubenAdam}. Future investigations should reveal if the category of geometric partial comodules could have a similar behaviour.

Concretely, after recalling the main features of the theory of geometric partial comodules over coalgebras in \S\ref{ssec:gparcom} and of the globalization construction in \S\ref{ssec:globgen}, we provide in \S\ref{ssec:noncoass} a general method to construct geometric partial comodules out of non-coassociative coactions (Theorem \ref{thm:constructgeometric}). This will allow us, in \S\ref{sec:abelian}, to easily show how a number of well-studied partial structures are, in fact, particular instances of geometric partial comodules. 
Our main result is presented in 
\S\ref{sec:abelian}, where we prove that globalization exists 
for geometric partial comodules over flat coalgebras in abelian monoidal categories (see Theorem \ref{thm:uffa}). In the subsequent sections, we apply it in a number of cases of interest. 
Namely, in \S\ref{ssec:hopf} we show that the partial modules over a Hopf algebra $H$, introduced in \cite{ParRep} in terms of partial representations of $H$, provide examples of geometric partial comodules in the opposite of the category of vector spaces. We also show that the so-called {\em standard dilation} of a partial representation as described in \cite{AlvesBatistaVercruysse} (see \cite{Abadie2} for the case of partial representations of groups) coincides with the globalization of the associated geometric partial comodule. After, in \S\ref{ssec:groups}, we verify that the globalization for partial representations of finite groups, as introduced and applied in \cite{MicheleWilliam}, can be recovered by our approach, too. 
In \S\ref{ssec:pcom} we prove that the partial comodules arising from corepresentations as defined in \cite{ParCorep} can be viewed as geometric partial comodules, this time in the category of vector spaces, and we obtain then a globalization theorem for these ``algebraic'' partial comodules, a construction that was completely unknown before. Finally, in \S\ref{sec:Hopf} we consider geometric partial comodules in the monoidal category of modules over a bialgebra, introducing in this way a suitable notion of ``Hopf partial comodules'' and establishing a globalization result for these, as well. As an application, we present a version of the fundamental theorem for these Hopf partial comodules and we consider the particular case of partially graded group representations.


\section{Preliminaries: Geometric partial comodules and globalization}\label{sec:preliminaries}

\subsection{Geometric partial comodules}\label{ssec:gparcom}

Let $\left(\cC,\otimes,\I\right)$ be a (strict) monoidal, locally small (also known as $\Set$-enriched) 
category with pushouts. For any object $X$ in $\cC$, we usually denote the identity morphism on $X$ again by $X$. Moreover, for any algebra $A$ and any coalgebra $H$ in $\cC$, we denote by $\Mod_A$ the category of (right) $A$-modules and by $\Com^H$ the category of (right) $H$-comodules.

A {\em partial comodule datum} over a coalgebra $(H,\Delta,\varepsilon)$ in $\cC$ is a cospan
\[
\begin{gathered}
\xymatrix@R=7pt{
X \ar[dr]_-{\rho_X} & & X\otimes H \ar@{->>}[dl]^(0.45){\pi_X} \\
 & X\bullet H & 
}
\end{gathered}
\]
in $\cC$ where $\pi_X$ is an epimorphism.
Any partial comodule datum induces canonically the following pushouts 
\begin{equation}\label{doublebullets}
\begin{gathered}
\xymatrix @!0 @R=37pt @C=55pt{
& X\ot H \ar@{->>}[dl]_-{\pi_X} \ar[dr]^-{\rho_X\ot H}\\
X\bul H \ar[dr]_(0.4){\rho_X\bul H}  && (X\bul H)\ot H \ar@{->>}[dl]^(0.4){\pi_{X\bul H}}\\
&(X\bul H)\bul H\ar@{}[u]|<<<{\pushout}
}
\end{gathered}
\qquad\text{and}\qquad
\begin{gathered}
\xymatrix @!0 @R=37pt @C=55pt{
& X\ot H \ar@{->>}[dl]_(0.55){\pi_X} \ar[dr]^(0.55){\quad (\pi_X \ot H)\circ (X\ot \Delta)} \\
X\bul H \ar[dr]_(0.4){X\bul \Delta} & & (X\bul H)\ot H\ar@{->>}[dl]^(0.45){\pi_{X,\Delta}} \\
& X\bul(H\bul H) \ar@{}[u]|<<<{\pushout} &
}
\end{gathered}
\end{equation}

\begin{definition}\cite{JoostJiawey}
Let $(H,\Delta,\varepsilon)$ be a coalgebra in $\cC$. A {\em geometric partial comodule} over $H$ is a partial comodule datum $(X,X\bul H,\pi_X,\rho_X)$ that satisfies the following conditions.
\begin{enumerate}[label=({GPC\arabic*}),leftmargin=1.7cm]
\item\label{item:QPC1} Counitality:
there exists a morphism $X\bul\varepsilon:X\bul H\to X$ which makes the following diagram commutative
\[
\begin{gathered}
\xymatrix @!0 @R=16pt @C=65pt{
X  \ar@/_2.5ex/[dddr]_-{\id_X} \ar[dr]^-{\rho_X} && X\ot H  \ar@{->>}[dl]_-{\pi_X}\ar@/^2.5ex/[dddl]^-{X\ot\varepsilon} \\
& X\bul H \ar[dd]_-{X\bul\varepsilon} & \\
 & & \\
& X.
}
\end{gathered}
\]
\item\label{item:QPC2} Geometric coassociativity:
there exists an isomorphism 
\[\theta:X\bul(H\bul H)\to (X\bul H)\bul H\] 
such that the following diagram commutes
\[
\xymatrix @R=20pt @C=35pt{
X \ar[r]^-{\rho_X} \ar[d]_-{\rho_X} & X\bul H \ar[r]^-{\rho_X\bul H} & (X\bul H)\bul H 
\\
X\bul H \ar[r]_-{X\bul \Delta} & X\bul(H\bul H) \ar[ur]_{\theta} & (X\bul H)\ot H. \ar@{->>}[l]^-{\pi_{X,\Delta}} \ar@{->>}[u]_-{\pi_{X\bul H}}
}
\]
\end{enumerate}

If $(X,X\bul H,\pi_X,\rho_X)$ and $(Y,Y\bul H,\pi_Y,\rho_Y)$ are geometric partial comodules, then a \emph{morphism of geometric partial comodules} is a pair $(f,f\bul H)$ of morphisms in $\cC$ with $f:X\to Y$ and $f\bul H: X\bul H \to Y\bul H$ such that the following diagram commutes
\begin{equation}\label{eq:shield}
\begin{gathered}
\xymatrix @!0 @R=18pt @C=60pt {
X \ar[dd]_-{f} \ar[dr]^-{\rho_X} & & X\ot H \ar@{->>}[dl]_-{\pi_X} \ar[dd]^-{f\ot H} \\ 
 & X\bul H \ar[dd]_-{f\bul H} &  \\
Y \ar[dr]_-{\rho_Y} & & Y \ot H \ar@{->>}[dl]^-{\pi_{Y}} \\
 & Y \bul H. & 
}
\end{gathered}
\end{equation}
We will often denote a geometric partial comodule $(X,X\bul H,\pi_X,\rho_X)$ simply by $X$ and a morphism as above simply by $f$. Moreover, we denote by $\gparcom{H}$ the category of geometric partial comodules over $H$ and their morphisms.
\end{definition}

\begin{remark}
Any usual global comodule $(Y,\delta_Y)$ over $H$ is naturally a geometric partial comodule $(Y,Y \ot H,\id_{Y \ot H},\delta_Y)$. In fact, $\Com^H$ is a full subcategory of $\gparcom{H}$ with associated embedding functor $\cI:\com{H} \to \gparcom{H}$.
\end{remark}


\subsection{Globalization for geometric partial comodules}\label{ssec:globgen}

Recall from \cite[Example 2.5]{JoostJiawey} that, for a given epimorphism $p : Y \to X$ in $\cC$, the pushout
\begin{equation}\label{eq:globcom}
\begin{gathered}
\xymatrix @!0 @R=23pt @C=60pt {
 & Y \ar[dr]^(0.55){\,\ (p\otimes H)\circ\delta} \ar@{->>}[dl]_-{p} & \\
 X \ar[dr]_-{\rho_X} & & X\otimes H \ar@{->>}[dl]^-{\pi_X} \\
 & X\bul H \ar@{}[u]|<<<{\pushout} &
}
\end{gathered}
\end{equation}
inherits a structure of partial comodule and $p$ becomes a morphism of partial comodules. We refer to this as the \emph{induced partial comodule} structure from $Y$ to $X$.

Naively speaking, the globalization of a partial comodule $X$ is a universal  $H$-comodule ``covering'' $X$ and such that the partial coaction is induced by the global one as above. 

\begin{definition}[{\cite[Definition 3.1]{Saracco-Vercruysse}}]\label{def:glob}
Given a partial comodule $(X,X\bul H,\pi_X,\rho_X)$, a \emph{globalization} for $X$ is a global comodule $(Y,\delta)$ with a morphism $p:Y\to X$ in $\cC$ such that
\begin{enumerate}[label=(GL\arabic*),ref=(GL\arabic*)]
\item\label{item:GL0} $p$ is a morphism of partial comodules, that is, the following diagram commutes:
\[
\begin{gathered}
\xymatrix{
Y \ar[d]_p \ar[r]^-{\delta_Y} & Y\ot H \ar[r]^-{p\ot H} & X\ot H \ar[d]^-{\pi_X}\\
X \ar[rr]_-{\rho_X} & & X\bul H.
}
\end{gathered}
\]
\item\label{item:GL3} $Y$ is universal with respect to this property, i.e.~there is a bijective correspondence
\[\com{H}(Z,Y) \to \gparcom{H}(\cI(Z),X), \qquad \eta \mapsto p \circ \eta.\]
\item\label{item:GL2} The corresponding diagram \eqref{eq:globcom} is a pushout square in $\cC$.
\end{enumerate}
We say that $X$ is {\em globalizable} if a globalization for $X$ exists and we denote by $\gparcom{H}_{gl}$ the full subcategory of $\gparcom{H}$ composed by the globalizable partial comodules. 
\end{definition}

It can be shown (see \cite[Lemma 3.2]{Saracco-Vercruysse}) that if $(Y,p)$ is a globalization of a partial comodule $X$, then $p:Y\to X$ is an epimorphism. Moreover, it follows from axiom \ref{item:GL3} that a globalization of a partial comodule $X$ is unique, whenever it exists. Thus, we may speak about \emph{the} globalization of $X$.

\begin{remark}\label{rem:ax3}
Axiom \ref{item:GL2} 
cannot be omitted in general, as it is not always necessarily satisfied (see, for instance, \cite[Example 3.6]{Saracco-Vercruysse}). Nevertheless, there are cases in which an object satisfying \ref{item:GL0} and the universal property \ref{item:GL3}, also satisfies \ref{item:GL2}. For instance, this is the case in $\op{\Set}$. For this reason, axiom \ref{item:GL2} has been often neglected in the literature (see e.g. \cite[Theorem 1.1]{Abadie}, \cite[Definition 2.5]{Megrelishvili}). 
\end{remark}

As already mentioned in Remark \ref{rem:ax3}, it is known that non-globalizable partial comodules may exist. This is the case, for instance, in the category $ (\ms{C})\algk$ of (commutative) algebras over a field $\K$ (see \cite[Corollary 3.7]{Saracco-Vercruysse}) or in the opposite $\op{\Top}$ of the category of topological spaces (see \cite[Proposition 3.2]{Saracco-Vercruysse1}). The aim of the present paper is to show that (under mild assumptions on the coalgebra $H$ or, more precisely, on its category of comodules) over an abelian monoidal category $\cC$ we have $\gparcom{H} = \gparcom{H}_{gl}$, obtaining in this way the optimum among the globalization results.

The following is the main result of  \cite{Saracco-Vercruysse}, Theorem 3.5, for the sake of the reader. 

\begin{theorem}\label{thm:globalization}
Let $H$ be a coalgebra in a monoidal category $\cC$ with pushouts. Then a geometric partial $H$-comodule $X=(X,X\bul H,\pi_X,\rho_X)$ is globalizable if and only if 
\begin{enumerate}[label=(\arabic*),ref=(\arabic*),leftmargin=1cm]
\item\label{item:glob1} the following equalizer exists in $\com{H}$:
\begin{equation}\label{eq:glob}
\begin{gathered}
\xymatrix@C=25pt{
(Y_X,\delta) \ar@<+0.2ex>[r]^-{\kappa} & (X\otimes H,X\ot \Delta) \ar@<+0.8ex>[rrr]^-{\rho_X\ot H} \ar@<-0.4ex>[rrr]_-{(\pi_X\ot H)\circ(X \otimes \Delta)} & & &  (X\bul H\otimes H, X\bul H\ot \Delta)
}
\end{gathered}
\end{equation}
\item\label{item:glob3} the following diagram is a pushout diagram in $\cC$:
\begin{equation}\label{eq:GXglob}
\begin{gathered}
\xymatrix @!0 @R=23pt @C=60pt {
 & Y_X \ar[dr]^-{\kappa} \ar@{->}[dl]_-{(X\ot \varepsilon)\circ \kappa} & \\
 X \ar[dr]_-{\rho_X} & & X\otimes H \ar@{->>}[dl]^-{\pi_X} \\
 & X\bul H. &
}
\end{gathered}
\end{equation}
\end{enumerate}
Moreover, if these conditions hold, then the morphism $\epsilon_X=(X\ot \varepsilon)\circ \kappa: Y_X \to X$ is an epimorphism in $\cC$,
$\kappa = (\epsilon_X \otimes H) \circ \delta$ and $(Y_X,\epsilon_X)$ is the globalization of $X$.
\end{theorem}

It deserves to be mentioned (see \ref{item:GL3}) that the assignment $X\mapsto Y_X$ induces a functor 
$ 
\cG:\gparcom{H}_{gl}\to \com{H}
$ 
which is right adjoint to the natural, fully faithful, inclusion functor $\cJ:\com{H}\to\gparcom{H}_{gl}$ obtained by corestriction of $\cI : \com{H} \to \gparcom{H}$. 

Let us conclude this section by an alternative description of globalizable partial comodules in terms of 
\emph{minimal proper covers}. Recall from \cite[\S2.2]{Saracco-Vercruysse} that a \emph{cover} is simply the datum $(Y,X,p)$ of a global $H$-comodule $(Y,\delta_Y)$, an object $X$ and an epimorphism $p : Y \to X$ in $\cC$ and recall that a morphism of covers $(Y,X,p) \to (Y',X',p')$ is the datum of a morphism of $H$-comodules $F : Y \to Y'$ and a morphism $f : X \to X'$ in $\cC$ such that $f \circ p = p'\circ F$. The cover $(Y,X,p)$ is said to be \emph{proper} if 
\[
(p\ot H)\circ \delta_Y:Y\to X\ot H
\]
is a monomorphism in $\com{H}$ (in which case we say that $Y$ is \emph{co-generated} by $X$) and a proper cover is said to be \emph{minimal} if it does not factor through another proper cover. More explicitly, if $(Y',X,p')$ is another proper cover such that $p=p'\circ q$ for some morphism $q:Y\to Y'$ in $\cC$, then $q$ is an isomorphism. 

\begin{theorem}[{\cite[Theorem 3.11]{Saracco-Vercruysse}}]\label{globcov}
Assume that the equalizer \eqref{eq:glob} exists in $\Com^{H}$ for every partial comodule $X$ (e.g., when $\Com^{H}$ is complete).
If $X$ is a partial comodule that has been induced by a global comodule, then $X$ is globalizable. That is, we have a functor 
\[\Ind:\Cov^H\to \gparcom{H}_{gl}\]
resulting from the induction construction, which has a fully faithful right adjoint ${\sf Gl}$ given by ${\sf Gl}(X)=(\cG(X),X,\epsilon_X)$. Moreover, for any globalizable partial comodule, ${\sf Gl}(X)$ is a minimal proper cover and the functors $\Ind$ and ${\sf Gl}$ induce an equivalence of categories
\[\xymatrix{\Cov^H_{pr,min} \ar@<.5ex>[rr]^-{\Ind} & \sim & \gparcom{H}_{gl} \ar@<.5ex>[ll]^-{\sf Gl}}.\]
\end{theorem}

\subsection{Non-coassociative coactions}\label{ssec:noncoass}

In this section,
we present a new general procedure that allows us to construct a wide class of examples of geometric partial comodules, out of \emph{non-coassociative coactions}.
Still, not all geometric partial comodules arise in this way, as Example \ref{ex:nass} shows. Hence, the theory of partial comodules cannot be simply reduced to the non-coassociative case.

Let $(H,\Delta,\varepsilon)$ be a coalgebra in the (strict) monoidal category with pushouts $\cC$.
Denote by $\ncom{H}$ the category of non-coassociative $H$-comodules, that is, $H$-comodules which are counital but not necessarily coassociative. Explicitly, an object in $\ncom{H}$ is an object $X$ in $\cC$ endowed with a morphism $\partial_X:X\to X\ot H$ in $\cC$ satisfying $(X\ot \varepsilon)\circ \partial_X=X$. Morphisms in $\ncom{H}$ are maps $f:X \to X'$ in $\cC$ such that $\partial_{X'} \circ f = (f \ot H) \circ \partial_X$.

Let $(X,\partial_X)$ be an object in $\ncom{H}$ and suppose that an object $X\bul H$ and a morphism $\pi_X:X\ot H\to X\bul H$ exist such that the following parallel morphisms are equal
\begin{equation}\label{buluniversal}
\xymatrix{
X \ar[rr]^-{\partial_X} && X\ot H \ar@<.5ex>[rr]^-{\partial_X\ot H} \ar@<-.5ex>[rr]_-{X\ot \Delta} && X\ot H\ot H \ar[rr]^-{\pi_X\ot H} &&(X\bul H)\ot H
}
\end{equation}
and such that the following universal property holds: for any other object $T$ and morphism $t:X\ot H\to T$ in $\cC$ such that
$$(t\ot H)\circ (\partial_X\ot H)\circ \partial_X = (t\ot H)\circ (X\ot \Delta)\circ \partial_X$$
there exists a unique morphism $u:X\bul H\to T$ such that $t=u\circ \pi$. In this case, we will (improperly) say that $(X \bul H, \pi_X)$ \emph{universally coequalizes} the arrows in \eqref{buluniversal}.
From the universal property one easily deduces (by using an argument similar to the case of coequalizers) that $\pi_X:X\ot H\to X\bul H$ is an epimorphism. Moreover, the following lemma and remark show that in many cases of interest, $\pi_X:X\ot H\to X\bul H$ is in fact a colimit.

\begin{lemma}\label{buliscolimit}
In the framework of the foregoing paragraph, the following statements hold.
\begin{enumerate}[leftmargin=1cm,ref=(\arabic*)]
\item\label{item:ncc1} If the family of morphisms $\{(X\bul H)\ot f ~|~ f:H\to \I\}$ is jointly monic, then $\pi_X:X\ot H\to X\bul H$ is a colimit in $\cC$.
Namely, it can be realized as the universal arrow coequalizing the following compositions
\[
\begin{gathered}
\xymatrix{
X\ar[r]^-{\partial_X} & X\ot H \ar@<.5ex>[rr]^-{\partial_X\ot H} \ar@<-.5ex>[rr]_-{X\ot \Delta} && X\ot H\ot H \ar[rr]^-{X\ot H\ot f} &&
X\ot H
}
\end{gathered}
\] 
for all $f\in {\sf Hom}(H,\I)$.
\item\label{item:ncc2} If $\cC$ is complete,
the monoidal unit $\I$ is a cogenerator in $\cC$ and the endofunctor $(X\bul H)\ot -$ preserves limits, then the condition of \ref{item:ncc1} holds.
\end{enumerate}
\end{lemma}

\begin{proof}
\ref{item:ncc1}.
Consider a small category $\cZ$ whose set of objects is $\{Z_0, Z_f ~|~ f\in {\sf Hom}(H,\I)\}$ and whose morphisms are the identity morphisms plus two additional morphisms $z^f_1,z^f_2:Z_f\to Z_0$ for each $f\in{\sf Hom}(H,\I)$. Define a functor $F:\cZ\to \cC$ acting on objects as
\[
F(Z_0)=X\ot H \qquad \text{ and } \qquad F(Z_f)=X \ \text{ for all } f\in {\sf Hom}(H,\I)
\]
and on (non-identity) morphisms as 
\[
F(z^f_1)=(X\ot H\ot f)\circ (\partial_X\ot H)\circ \partial_X \quad \text{ and } \quad F(z^f_2)=(X\ot H\ot f)\circ (X\ot \Delta)\circ \partial_X.
\] 
Then the colimit of $F$ is a universal arrow $c:X\ot H\to C$ which coequalizes the following compositions for all $f\in {\sf Hom}(H,\I)$
\[
\begin{gathered}
\xymatrix{
X\ar[rr]^-{\partial_X} && X\ot H \ar@<.5ex>[rr]^-{\partial_X\ot H} \ar@<-.5ex>[rr]_-{X\ot \Delta} && X\ot H\ot H \ar[rr]^-{X\ot H\ot f} &&
X\ot H \ar[rr]^-{c} && C
}
\end{gathered}.
\]
We can rewrite the above diagram as
\[
\begin{gathered}
\xymatrix{
X\ar[rr]^-{\partial_X} && X\ot H \ar@<.5ex>[rr]^-{\partial_X\ot H} \ar@<-.5ex>[rr]_-{X\ot \Delta} && X\ot H\ot H \ar[rr]^-{c\ot H} &&
C\ot H \ar[rr]^-{C\ot f} && C
}
\end{gathered}.
\]
If the family of morphisms $\{C\ot f \mid f: H \to \I\}$ is jointly monic, then it is equivalent to require that the following parallel morphisms are equal
\[
\begin{gathered}
\xymatrix{
X\ar[rr]^-{\partial_X} && X\ot H \ar@<.5ex>[rr]^-{\partial_X\ot H} \ar@<-.5ex>[rr]_-{X\ot \Delta} && X\ot H\ot H \ar[rr]^-{c\ot H} &&
C\ot H
}
\end{gathered}.
\]
This is exactly the universal property of $(X\bul H,\pi_X)$.

\ref{item:ncc2}. Recall that, being $\cC$ complete and locally small, $\I$ is a cogenerator if and only if for every object $A$ in $\cC$, the family $\{f \in \cC(A,\I)\}$ is jointly monic, if and only if the canonical map $A \to \prod_{f \in \cC(A,\I)}\I_f$ is a monomorphism. Therefore, \ref{item:ncc2} follows directly from the definition of a cogenerator, in combination with the fact that, since being monic in a complete category is a limit condition (in this setting, a morphism $f:A \to B$ is monic if and only if $(A,\id_A,\id_A)$ is the pullback of $A \xrightarrow{f} B \xleftarrow{f} A$), a functor preserving limits preserves also jointly monic families.
\end{proof}

\begin{remark}\label{rem:examples}
The condition of Lemma~\ref{buliscolimit}\ref{item:ncc2} (and hence of \ref{item:ncc1}) is satisfied, for example, in $\cC=\op{\vectk}$, in $\cC=\op{\Rmod{k}}$ for $k$ a commutative ring, in $\cC = (\op{\Set},\times,\{*\})$ and in $\cC = (\op{\Top},\times,\{*\})$.
Condition \ref{item:ncc1} of Lemma \ref{buliscolimit} is also satisfied in $\cC = \vectk$, because of bases. If $V$ is a vector space with basis $\{e_i\mid i \in I\}$, then the dual elements $\{e_i^* \mid i \in I\}$ form a jointly monic family, because the linear map $\Phi:V \to \prod_{i \in I}\K, v \mapsto \left(e_{i}^*(v)\right)_{i \in I},$ is a monomorphism. In fact, $\Phi$ is the composition of the isomorphism $\phi: V \to \bigoplus_{i \in I}\K$ (by definition of basis) with the canonical inclusion of $\bigoplus_{i \in I}\K$ into $\prod_{i \in I}\K$. Now, for every vector space $T$, the collection $\cB \coloneqq \{T \ot e_i^*: T \ot V \to T \mid  i \in I\}$ gives rise to a unique linear map $\Psi:T \ot V \to \prod_{i \in I} T$ which fits into the following commutative diagram
\[
\xymatrix @R=15pt @C=25pt {
T \ot V \ar[r]^-{T \ot \phi}_-{\cong} \ar[d]_-{\Psi} & {\displaystyle T \ot \left(\bigoplus_{i \in I}\K\right)} \ar[d]^-{\cong} \\
\prod_{i \in I}T & {\displaystyle \bigoplus_{i \in I}T} \ar[l]^-{\subseteq}
}
\]
and hence $\cB$ is still a jointly monic family. In particular, for every $V,T$ in $\vectk$, the family $\{T \ot f \mid f : V \to \K\}$ is jointly monic.
\end{remark}

\begin{theorem}\label{thm:constructgeometric}
Let $(H,\Delta,\varepsilon)$ be a coalgebra in $\cC$ and $(X,\partial_X)\in \ncom{H}$ a counital non-coassociative $H$-comodule.
If an arrow $\pi_X:X\ot H\to X\bul H$ that universally coequalizes the arrows in \eqref{buluniversal} exists, then $(X,X\bul H,\pi_X,\pi_X\circ \partial_X)$ is a geometric partial $H$-comodule.
\end{theorem}

For the sake of future reference, recall from Remark \ref{rem:examples} that $\vectk$, $\op{\vectk}$, $\op{\Rmod{k}}$, $\op{\Set}$ and $\op{\Top}$ are all monoidal categories for which Theorem \ref{thm:constructgeometric} can be applied.

\begin{proof}[Proof of Theorem \ref{thm:constructgeometric}.]
Let us first prove that $(X\bul H)\bul H$ and $X\bul (H\bul H)$, constructed as in 
\eqref{doublebullets}, are isomorphic. 
Applying the coassociativity of $H$ and the naturality of the tensor product in the unlabelled equalities, we find
\begin{align*}
(\pi_{X,\Delta} \ot H) & \circ (\pi_X\ot H\ot H)\circ (\partial_X\ot H\ot H)\circ (\partial_X\ot H)\circ \partial_X\\
&\stackrel{\eqref{buluniversal}}{=}(\pi_{X,\Delta}\ot H)\circ (\pi_X\ot H\ot H)\circ (X\ot \Delta \ot H)\circ (\partial_X\ot H)\circ \partial_X\\
&\stackrel{\eqref{doublebullets}}{=} ((X\bul \Delta)\ot H) \circ (\pi_X\ot H) \circ (\partial_X\ot H)\circ \partial_X\\
&\stackrel{\eqref{buluniversal}}{=} ((X\bul \Delta)\ot H) \circ (\pi_X\ot H) \circ (X\ot \Delta)\circ \partial_X\\
&\stackrel{\eqref{doublebullets}}{=} (\pi_{X,\Delta}\ot H)\circ (\pi_X\ot H\ot H)\circ (X\ot \Delta \ot H)\circ (X\ot \Delta)\circ \partial_X\\
&\stackrel{\phantom{(*)}}{=} (\pi_{X,\Delta}\ot H)\circ (\pi_X\ot H\ot H)\circ (X\ot H\ot \Delta)\circ (X\ot \Delta)\circ \partial_X\\
&\stackrel{\phantom{(*)}}{=} (\pi_{X,\Delta}\ot H)\circ ((X\bul H)\ot \Delta)\circ (\pi_X\ot H)\circ (X\ot \Delta)\circ \partial_X\\
&\stackrel{\eqref{buluniversal}}{=} (\pi_{X,\Delta}\ot H)\circ ((X\bul H)\ot \Delta)\circ (\pi_X\ot H)\circ (\partial_X\ot H)\circ \partial_X\\
&\stackrel{\phantom{(*)}}{=} (\pi_{X,\Delta}\ot H)\circ  (\pi_X\ot H\ot H)\circ (\partial_X\ot H\ot H) \circ (X\ot \Delta)\circ \partial_X
\end{align*}
Hence, by the universal property of $(X\bul H,\pi_X)$, there exists a (unique) morphism $u:X\bul H\to X\bul (H\bul H)$ such that 
$$u\circ \pi_X=\pi_{X,\Delta}\circ (\pi_X\ot H)\circ (\partial_X\ot H).$$
Consequently, by the universal property of the pushout $(X\bul H)\bul H$, there exists a (unique) morphism of cospans $\alpha:
(X\bul H)\bul H\to X\bul (H\bul H)$.
Similarly, one obtains that
\begin{multline*}
(\pi_{X\bul H}\ot H)\circ (\pi_X\ot H\ot H)\circ (X\ot \Delta\ot H) \circ (\partial_X \ot H)\circ \partial_X \\
= (\pi_{X\bul H}\ot H)\circ (\pi_X\ot H\ot H) \circ (X \ot \Delta\ot H)\circ (X\ot \Delta)  \circ \partial_X
\end{multline*}
and from this one concludes that there exists a (unique) morphism of cospans $\beta:X\bul (H\bul H)\to (X\bul H)\bul H$. Since the category of cospans over a fixed pair of domains forms a poset, it follows directly from their existence that $\alpha$ and $\beta$ are mutual inverses and hence $(X\bul H)\bul H$ and $X\bul (H\bul H)$ are isomorphic as cospans.

By definition of $X\bul H$ we have that coassociativity holds in $(X\bul H)\ot H$ and hence also in $(X\bul H)\bul H=X\bul (H\bul H)$. Thus, $X$ satisfies axiom \ref{item:QPC2}.
Concerning counitality, let us observe that 
$$(X\ot \varepsilon\ot H)\circ (\partial_X\ot H)\circ \partial_X = \partial_X = (X\ot \varepsilon\ot H)\circ (X\ot \Delta)\circ \partial_X$$
as maps $X\to X\ot H$, by hypothesis on $\partial_X$. By the universal property of $X\bul H$ there exists a unique map $X\bul\varepsilon:X\bul H\to X$ such that $(X\bul\varepsilon)\circ \pi_X = X\ot\varepsilon$. Therefore axiom \ref{item:QPC1} holds too. Thus, $X$ is a geometric partial comodule.
\end{proof}

The converse of Theorem \ref{thm:constructgeometric} does not hold in general, that is to say, there exist geometric partial comodules $(X,X\bul H,\pi_X,\rho_X)$ in $\cC$ such that $\rho_X = \pi_X \circ \partial_X$ for some non-coassociative $\partial_X : X \to X \ot H$, but $(X\bul H,\pi_X)$ does not satisfy the universal property. In fact, for every $(X,\partial_X)$ in $\ncom{H}$ 
the composition $\rho_X \coloneqq (X \otimes \varepsilon) \circ \partial_X = \id_X$ defines a geometric partial comodule structure $(X,X,X \otimes \varepsilon, \id_X)$ on $X$ (it is enough to observe that requiring $X \ot \varepsilon$ to be an epimorphism suffices to prove \cite[Proposition 2.20]{JoostJiawey}). However, $(X,X \ot \varepsilon)$ rarely universally coequalizes the arrows in \eqref{buluniversal}, as the following example shows.

\begin{example}\label{ex:boh}
Consider $\cC = (\Set,\times,\{*\})$. It is well-known (see, for instance, \cite[Lemma 1.3]{Stef-Turaev}) that the diagonal map $\Delta : H \to H \times H, x \mapsto (x,x),$ and the (unique) map $\varepsilon: H \to \{*\}, x \mapsto *,$ make of any non-empty set $H$ a coalgebra in $\cC$ and that this one is the unique coassociative and counital coalgebra structure one can equip $H$ with. In addition, if $X$ is any set and $f : X \to H$ is any function, then the assignment
\[
\partial_X : X \to X \times H, \qquad x \mapsto (x,f(x))
\]
defines a counital coaction (which is automatically coassociative) and all the counital coactions are of this form. 
As a consequence, any counital coaction $\partial_X$ in $\Set$ is in fact coassociative and hence the pair $(X \times H, \id)$ universally coequalizes the arrows in \eqref{buluniversal}. Therefore, all the geometric partial comodules provided by Theorem \ref{thm:constructgeometric} in $\cC = \Set$ are global, but not all the geometric partial comodules are so. For instance, for any non-empty sets $H \neq \{*\}$ and $X$, the trivial geometric partial $H$-comodule structure
\[
\xymatrix @!0 @R=25pt @C=45pt{
X \ar@{=}[dr] & & X \times H \ar@{->>}[dl]^-{X \times \varepsilon} \\
 & X & 
}
\]
on $X$ is not global, whence it cannot come from a universal pair $(X \bul H,\pi_X)$.
\end{example}

It follows from the proof of the next corollary that whenever the underlying category $\cC$ is such that a universal object $(X\bul H,\pi_X)$ coequalizing \eqref{buluniversal} exists for every non-coassociative comodule $(X,\partial_X)$, then the construction of Theorem \ref{thm:constructgeometric} becomes functorial.

\begin{corollary}[of Theorem \ref{thm:constructgeometric}]\label{cor:ncoass}
Let $(H,\Delta,\varepsilon)$ be a coalgebra in a cocomplete monoidal category $\cC$ for which the family $\{Y \ot f : Y \ot H \to Y \mid f \in \cC(H,\I)\}$ is jointly monic for every $Y$ in $\cC$. 
Then there exists a faithful functor 
$$\cP : \ncom{H} \to \gparcom{H}$$
which commutes with the forgetful functors to $\cC$.
\end{corollary}

\begin{proof}
By cocompleteness and by Lemma \ref{buliscolimit}, we know that for any object $(X,\partial_X)$ in $\ncom{H}$ there exists a morphism $\pi_X:X\ot H\to X\bul H$, universally coequalizing \eqref{buluniversal}. Hence, it follows from Theorem \ref{thm:constructgeometric} that the functor $\cP$ can be defined at the level of objects by $\cP(X,\partial_X)=(X,X\bul H, \pi_X, \pi_X\circ \partial_X)$. 
Let us show that for any morphism $f:(X,\partial_X)\to (X',\partial_{X'})$ in $\ncom{H}$, the underlying morphism in $\cC$ is also a morphism of geometric partial comodules from $\cP(X,\partial_X)$ to $\cP(X',\partial_{X'})$. 

Indeed, by using the defining property of morphisms in $\ncom{H}$ we find 
\begin{align*}
&  (\pi_{X'} \otimes H) \circ (f \otimes H \otimes H) \circ (\partial_X \otimes H) \circ \partial_X = (\pi_{X'} \otimes H) \circ (\partial_{X'} \otimes H)  \circ (f \otimes H) \circ \partial_X \\
& = (\pi_{X'} \otimes H) \circ (\partial_{X'} \otimes H) \circ \partial_{X'}  \circ f \stackrel{\eqref{buluniversal}}{=} (\pi_{X'} \otimes H) \circ (X' \otimes \Delta) \circ \partial_{X'}  \circ f \\
& = (\pi_{X'} \otimes H) \circ (X' \otimes \Delta) \circ (f \otimes H) \circ \partial_X =  (\pi_{X'} \otimes H) \circ (f \otimes H \otimes H) \circ (X \otimes \Delta) \circ \partial_X.
\end{align*}
Then, by the universal property of $(X\bul H,\pi_X)$, there exists a unique morphism $f\bul H: X\bul H \to X'\bul H$ such that
\begin{equation}\label{eq:pif}
(f \bul H) \circ \pi_X = \pi_{X'} \circ (f \otimes H).
\end{equation}
By definition of $\rho_X = \pi_X \circ \partial_X$, we have that
\[
(f \bul H) \circ \rho_X = (f \bul H) \circ\pi_X \circ \partial_X \stackrel{\eqref{eq:pif}}{=} \pi_{X'} \circ (f \otimes H)\circ \partial_X = \pi_{X'}\circ \partial_{X'} \circ f = \rho_{X'} \circ f
\]
and hence
\[
\xymatrix @!0 @R=16pt @C=60pt {
X \ar[dd]_-{f} \ar[dr]^-{\rho_X} & & X\ot H \ar[dl]_-{\pi_X} \ar[dd]^-{f\ot H} \\ 
 & X\bul H \ar[dd]_-{f\bul H} &  \\
X' \ar[dr]_-{\rho_{X'}} & & X' \ot H \ar[dl]^-{\pi_{X'}} \\
 & X' \bul H & 
}
\]
commutes, making of $(f,f\bul H)$ a morphism of geometric partial comodules.

Obviously, this turns $\cP$ into a functor commuting with the forgetful functors, which is then faithful.
\end{proof}

Before continuing, let us show that 
the functor $\cP$ of Corollary \ref{cor:ncoass} is in general neither full, nor essentially surjective on objects. 
Hence, even in those favourable cases in which every non-coassociative comodule gives rise to a geometric partial comodule, still not every geometric partial comodule, nor every geometric partial comodule morphism, can be obtained in this way, indicating that the category of geometric partial comodules has a richer structure.

\begin{example}
Consider the $2$-dimensional complex coalgebra $H$ generated by a grouplike element $g$ and a $g$-primitive element $x$. That is, $H=\C g\oplus \C x$ with comultiplication $\Delta$ and counit $\varepsilon$ defined by
\begin{eqnarray*}
& \Delta(g)=g\ot g, & \qquad \varepsilon(g)=1, \\
& \Delta(x)=g\ot x + x\ot g, & \qquad \varepsilon(x)=0.
\end{eqnarray*}
Now consider $V=\C g\oplus \C x$, which is the same complex vector space, but which we endow with a coaction $\partial:V\to V\ot H$ defined by
\[
\partial(x)=x\ot g \qquad \text{and} \qquad \partial(g)=g\ot g+ g\ot x.
\]
One can easily verify that this defines a counital but non-coassociative $H$-comodule. By applying Lemma~\ref{buliscolimit}\ref{item:ncc1} and Theorem~\ref{thm:constructgeometric}, we obtain a geometric partial comodule structure on $V$ where $V\bul H = (V\ot H)/\C(g\ot x)$, $\pi:V\ot H\to (V\ot H)/\C(g\ot x)$ is the natural projection and $\rho=\pi\circ \partial:V\to V\bul H$ is the partial coaction which is given by
$$\rho(g)=\ol{g\ot g}, \qquad \rho(x)=\ol{x\ot g}.$$
Now consider the $\C$-linear map
$$f:V\to V,\qquad f(x)=g=f(g).$$
Then one easily checks that 
$$(f\ot H)\partial(x)=g\ot g \neq \partial f(x)=g\ot g+g\ot x.$$
Therefore, $f$ is not a morphism in $\ncom{H}$. On the other hand, $(f\ot H)(g\ot x)=g\ot x$ hence $f\bul H:V\bul H\to V\bul H$ is well-defined and
\begin{eqnarray*}
\big((f\bul H)\circ \rho\big)(g)=\ol{g\ot g} = \rho\circ f(g), \\
\big((f\bul H)\circ \rho\big)(x)=\ol{x\ot g} = \rho\circ f(x),
\end{eqnarray*}
hence $f$ is a morphism in $\gparcom{H}$. Thus, the functor $\cP$ from Corollary~\ref{cor:ncoass} is not full. 
\end{example}

\begin{example}\label{ex:nass}
In $\cC = \op{\Ab}$ consider the geometric partial $\Q$-comodule structure on $\Z$ induced by the multiplication of $\Q$, that is to say, consider the pushout in $\cC$
\begin{equation}\label{eq:nass}
\begin{gathered}
\xymatrix @!0 @R=22pt @C=55pt {
 & \Q & \\
 \Q \otimes_\Z \Z \ar[ur]^(0.55){m_\Q\circ(\Q\otimes_Z \iota_\Z)\,\ } & & \Z \ar@{_{(}->}[ul]_-{\iota_\Z} \\
 & \Q\bul\Z. \ar@{}[u]|<<<{\pushout} \ar[ur]_-{\rho_\Z} \ar@{_{(}->}[ul]^-{\pi_\Z} &
}
\end{gathered}
\end{equation}
It is well-known that, in fact, $\Q \otimes_\Z \Q \cong \Q$ via the multiplication $m_\Q$ and hence $m_\Q \circ (\Q \otimes_\Z \iota_\Z) = \id_\Q$. Moreover, since it is clear that the pushout in $\op{\Ab}$ of the pair $(\id_\Q,\iota_\Z)$ is nothing else than $(\Z,\iota_\Z,\id_\Z)$, diagram \eqref{eq:nass} becomes
\[
\begin{gathered}
\xymatrix @!0 @R=19pt @C=33pt {
 & \Q & \\
 \Q \ar@{=}[ur] & & \Z \ar@{_{(}->}[ul]_-{\iota_\Z} \\
 & \Z. \ar@{}[u]|<<<{\pushout} \ar@{=}[ur] \ar@{_{(}->}[ul]^-{\iota_\Z} &
}
\end{gathered}
\]
Since there does not exist any non-zero morphism of abelian groups $\Q \to \Z$, it follows that the geometric partial $\Q$-comodule $(\Z,\Z,\iota_\Z,\id_\Z)$ in $\cC = \op{\Ab}$ cannot be obtained from a non-coassociative comodule structure $(X,\partial_X)$ on $X$ as in Theorem \ref{thm:constructgeometric}. 
Therefore, $\cP:\mathsf{NCom}^\Q \to \mathsf{gPCom}^\Q$ from Corollary \ref{cor:ncoass} is not essentially surjective on objects. 
\end{example}

Let us conclude this section by showing that, under mild conditions, the geometric partial comodules arising from non-coassociative coactions are globalizable.

\begin{proposition}
\label{ex:globnoncoass} 
Let $(X,\partial_X)\in \ncom{H}$ be a non-coassociative counital coaction and assume that an arrow $\pi_X : X \ot H \to X \bul H$ universally coequalizing \eqref{buluniversal} exists. Assume also that the equalizer $\big((Y,\delta),\kappa\big)$ as in \eqref{eq:glob} 
exists in $\com{H}$ and that it is preserved by the forgetful functor to $\cC$. 
Then, $(Y,\delta)$ becomes the globalization of the geometric partial comodule $(X,X\bul H,\pi_X,\rho_X)$. 
\end{proposition}

\begin{proof}
In fact, $\partial_X$ equalizes $\big(\rho_X \ot H, (\pi_X \ot H)\circ (X \ot \Delta)\big)$ by definition of $\pi_X$ and hence there exists a unique morphism $\sigma:X \to Y$ in $\cC$ such that $\kappa \circ \sigma = \partial_X$. By remembering that $\pi_X \circ \kappa = \rho_X \circ \epsilon_X$, the latter entails that 
$\rho_X \circ \epsilon_X \circ \sigma = \pi_X \circ \kappa \circ \sigma = \pi_X \circ \partial_X = \rho_X$ 
and so $\epsilon_X \circ \sigma = \id_X$.  
Furthermore, if $X \xrightarrow{f} Z \xleftarrow{g} X \ot H$ is a diagram in $\cC$ such that $f \circ \epsilon_X = g \circ (\epsilon_X \ot H) \circ \delta$, then
\begin{equation}\label{eq:somerels}
f \circ \epsilon_X = g \circ \kappa \qquad \text{and} \qquad f = f \circ \epsilon_X \circ \sigma = g \circ \kappa \circ \sigma = g \circ \partial_X,
\end{equation}
whence
\begin{multline*}
(g \ot H) \circ (X \ot \Delta) \circ \partial_X = (g \ot H) \circ (X \ot \Delta) \circ \kappa \circ \sigma = (g \ot H) \circ (\kappa \ot H) \circ \delta \circ \sigma \\
\stackrel{\eqref{eq:somerels}}{=} (f \ot H) \circ (\epsilon_X \ot H) \circ \delta \circ \sigma = (f \ot H) \circ \kappa \circ \sigma = (f \ot H) \circ \partial_X \stackrel{\eqref{eq:somerels}}{=} (g \ot H) \circ (\partial_X \ot H) \circ \partial_X,
\end{multline*}
where in the second equality we used the fact that $\kappa$ is a morphism of $H$-comodules. By the universal property of $\pi_X$, there exists a unique $\tau: X \bul H \to Z$ such that $\tau \circ \pi_X = g$ and so, as a consequence, $\tau \circ \rho_X = \tau \circ \pi_X \circ \partial_X = g \circ \partial_X \stackrel{\eqref{eq:somerels}}{=} f$. Summing up, diagram \eqref{eq:GXglob} is a pushout square and so the conditions of Theorem \ref{thm:globalization} are satisfied.
\end{proof}


\section{Globalization in abelian monoidal categories}\label{sec:abelian}


\subsection{Geometric partial comodules in abelian categories}\label{ssec:abelian}

Assume that $\cC$ is an abelian monoidal category  
and that $(H,\Delta,\varepsilon)$ is any coalgebra in $\cC$.
Our first aim is to establish a general criterion for the existence of globalization, which will then allow us to conclude that geometric partial comodules are always globalizable if the coalgebra is flat.  
To this aim, let us recall the following well-known fact.

\begin{lemma}[{\cite[Note to \S2.4 at page 34]{Popescu}}]\label{lemma:A1}
Consider a square
\[
\xymatrix @C=16pt @R=4pt{
 & A \ar[dl]_-{g} \ar[dr]^-{f} & \\
B \ar[dr]_-{k} & & C \ar[dl]^-{h} \\
 & D & 
}
\]
where $f$ is a monomorphism and $h$ is an epimorphism. Then in an abelian category, this is a pushout (cocartesian) square if and only if it is a pullback (cartesian) square.
\end{lemma}

\begin{proposition}\label{prop:superthm}
Assume that $(X,X\bul H,\pi_X,\rho_X)$ is a partial comodule datum satisfying the counitality condition \ref{item:QPC1}. Consider the pullback
\[
\begin{gathered}
\xymatrix @!0 @R=23pt @C=60pt{
 & T \ar[dl]_-{\varpi} \ar[dr]^-{\lambda} \ar@{}[d]|<<<{\pullback} & \\
X \ar[dr]_-{\rho_X} & & X\otimes H \ar@{->>}[dl]^-{\pi_X} \\
 & X\bullet H
}
\end{gathered}
\]
in $\cC$. Then $(X,X\bul H,\pi_X,\rho_X)$ is a geometric partial $H$-comodule if and only if
\begin{equation}\label{eq:abelian0}
\begin{gathered}
\xymatrix@C=25pt{
T \ar@<+0.2ex>[r]^-{\lambda} & X\otimes H \ar@<+0.8ex>[rrr]^-{\rho_X\ot H} \ar@<-0.4ex>[rrr]_-{(\pi_X\ot H)\circ(X \otimes \Delta)} & & & X\bul H\otimes H
}
\end{gathered}
\end{equation}
is an equalizer in $\cC$.
\end{proposition}

\begin{proof}
We begin by proving that $X$ is a geometric partial comodule if and only if 
\begin{equation}\label{eq:abelian}
\left(\rho_X \otimes H\right) \circ \lambda = \left(\pi_X \otimes H\right) \circ \left(X \otimes \Delta\right) \circ \lambda.
\end{equation}
Let us consider the following diagrams
\begin{equation}\label{eq:abelian2}
\newdir{ >}{{}*!/-10pt/@{>}}
\begin{gathered}
\xymatrix @C=40pt{
T \ar@{}[dr]|{(t)} \ar[d]_-{\varpi} \ar@{ >->}[r]^-{\lambda} \ar@{}[dr]|(0.2){\big\lrcorner} & X \ot H \ar@{}[dr]|{(a)} \ar@{->>}[d]^-{\pi_X} \ar@{ >->}[r]^-{\rho_X \ot H} & X\bul H \ot H \ar@{->>}[d]^-{\pi_{X \bul H}} \\
X \ar@{ >->}[r]_-{\rho_X} & X \bul H \ar@{ >->}[r]_-{\rho_X \bul H} & (X \bul H) \bul H \ar@{}[ul]|(0.2){\big\ulcorner}
}
\end{gathered}
\end{equation}
and
\begin{equation}\label{eq:abelian3}
\newdir{ >}{{}*!/-10pt/@{>}}
\begin{gathered}
\xymatrix @C=30pt{
T \ar@{}[dr]|{(t)} \ar@{ >->}[r]^-{\lambda} \ar[d]_-{\varpi} \ar@{}[dr]|(0.2){\big\lrcorner} & X \ot H \ar[rr]^-{(\pi_X \ot H) \circ (X \ot \Delta)} \ar@{->>}[d]^-{\pi_X} & & X \bul H \ot H \ar@{->>}[d]^{\pi_{X,\Delta}} \\
X \ar@{ >->}[r]_-{\rho_X} & X \bul H \ar@{ >->}[rr]_-{X \bul \Delta} & & X \bul (H \bul H) \ar@{}[ul]|(0.2){\big\ulcorner}
}
\end{gathered}
\end{equation}
where, by Lemma \ref{lemma:A1}, $(t)$ is also a pushout square and $(a)$ is also a pullback square. If \eqref{eq:abelian} holds, then $(X \bul H) \bul H$ and $X\bul (H \bul H)$ are pushouts of the same diagram and hence there exists a unique isomorphism $\theta:X \bul (H \bul H) \to (X \bul H) \bul H$ such that
\[\theta \circ (X \bul \Delta) \circ \rho_X = (\rho_X \bul H) \circ \rho_X \qquad \text{and} \qquad \theta \circ \pi_{X,\Delta} = \pi_{X \bul H},\]
which is the geometric coassociativity \ref{item:QPC2}. Conversely, assume that \ref{item:QPC2} is satisfied. Then, in view of the commutativity of \eqref{eq:abelian3} and by the universal property of $T$ as a pullback of \eqref{eq:abelian2}, there exists a unique morphism $\sigma: T \to T$ such that
\[\varpi \circ \sigma = \varpi \qquad \text{and} \qquad (\rho_X \ot H) \circ \lambda \circ \sigma = (\pi_X \ot H) \circ (X \ot \Delta) \circ \lambda.\]
By applying $X \bul \varepsilon \ot H$ to both sides of the latter equality, we conclude that $\lambda \circ \sigma = \lambda$ and, being $\lambda$ a monomorphism, that $\sigma =\id_T$.

Now, let us show that \eqref{eq:abelian} holds if and only if \eqref{eq:abelian0} is an equalizer. Since one implication is trivial, let us focus on the other.
Call $(E,e)$ the equalizer of $\rho_X\ot H$ and $(\pi_X\ot H)\circ (X\ot \Delta)$ in $\cC$. If \eqref{eq:abelian} holds, then $\lambda$ equalizes $\rho_X\ot H$ and $(\pi_X\ot H)\circ (X\ot \Delta)$, whence there exists a unique morphism $\tau: T \to E$ such that $e \circ \tau = \lambda$. On the other hand, since
\begin{multline*}
\rho_X \circ (X \ot \varepsilon) \circ e = (X \bul H \ot \varepsilon)\circ (\rho_X \ot H) \circ e \\
= (X \bul H \ot \varepsilon)\circ (\pi_X \ot H) \circ (X \ot \Delta) \circ e = \pi_X \circ e,
\end{multline*}
there exists a unique morphism $\tau': E \to T$ such that $\lambda\circ \tau' = e$ and $\varpi\circ \tau' = (X \ot \varepsilon)\circ e$. Being $\lambda$ a monomorphism in $\cC$, from $\lambda \circ \tau'\circ \tau = e \circ \tau = \lambda$ it follows that $\tau'\circ \tau = \id_{T}$. Being $e$ a monomorphism in $\cC$, from $e \circ \tau \circ \tau' = \lambda \circ \tau' = e$ it follows that $\tau \circ \tau' = \id_E$ as well, thus proving that $E\cong T$.
\end{proof}

Let $(X,X\bul H,\pi_X,\rho_X)$ be a geometric partial $H$-comodule in $\cC$ and let $(Y_X,\delta)$ be the global comodule from Theorem \ref{thm:globalization}\ref{item:glob1}, i.e. $Y_X$ is the equalizer of the diagram \eqref{eq:glob} in $\com{H}$. Remark that by Proposition \ref{prop:superthm}, the object $T$ is the equalizer of the same diagram \eqref{eq:abelian0}, but in the underlying abelian category $\cC$. Consequently, the universal property of the equalizer induces a morphism $\xi:Y_X \to T$ in $\cC$. With this notation we have the following result.

\begin{corollary}\label{cor:globAb}
Let $(X,X\bul H,\pi_X,\rho_X)$ be a geometric partial $H$-comodule in $\cC$ and let $(Y_X,\delta)$ be the equalizer of the diagram \eqref{eq:glob} in $\com{H}$. 
If the canonical morphism $\xi:Y_X \to T$ 
is an epimorphism in $\cC$, then $Y_X$ is the globalization of $X$.
\end{corollary}

\begin{proof}
Consider the commutative diagram
\[
\newdir{ >}{{}*!/-10pt/@{>}}
\xymatrix @!0 @C=60pt @R=35pt{
 & Y_X \ar[ddl]_-{\epsilon_X} \ar[ddr]^-{\kappa} \ar@{->>}[d]^-{\xi} & \\
 & T \ar[dl]^-{\varpi} \ar@{ >->}[dr]_-{\lambda}  \ar@{}[d]|<<<{\pullback} & \\
X \ar@{ >->}[dr]_-{\rho_X} & & X \ot H \ar@{->>}[dl]^-{\pi_X} \\
 & X \bul H \ar@{}[uu]|-{(a)} & 
}
\]
By Lemma \ref{lemma:A1}, the square $(a)$ is a pushout square, and since $\xi$ is an epimorphism, it follows that the outer square is a pushout as well, making of $Y_X$ the globalization as claimed. 
\end{proof}

From Corollary \ref{cor:globAb} it is clear that to know if geometric partial comodules are globalizable, we should be able to compare equalizers in $\Com^H$ and $\cC$. We call an object $C$ in an abelian monoidal category $\cC$ {\em left flat} if the endofunctor $-\ot C:\cC\to\cC$ preserves equalizers (equivalently, $- \ot C$ is left exact, i.e.~it preserves finite limits, since the preservation of finite (bi)products is automatic for additive functors).
Let us therefore recall the following known (folklore) facts (the proof is based on \cite[3.4]{Wis}).

\begin{proposition}\label{Hflat} Let $H$ be a coalgebra in an abelian monoidal category $\cC$. Then the category $\com{H}$ has finite colimits and finite biproducts, and the forgetful functor $\cF:\com{H}\to \cC$ preserves them. Moreover, if one of the following equivalent conditions holds:
\begin{enumerate}[label=(\arabic*),ref=(\arabic*)]
\item\label{item:abelian1} $\cF$ preserves equalizers;
\item\label{item:abelian3} $\cF$ creates equalizers;
\item\label{item:abelian4} $H$ is left flat in $\cC$;
\end{enumerate}
then $\com{H}$ is abelian as well.  In addition, if $- \ot H$ preserves coequalizers (equivalently, it is right exact, i.e.~it preserves finite colimits -- this is the case, for example, when $\cC$ is closed monoidal), then any of the foregoing conditions \ref{item:abelian1} - \ref{item:abelian4} is equivalent to:
\begin{enumerate}[resume*]
\item\label{item:abelian2} $\cF$ preserves monomorphisms.
\end{enumerate}
\end{proposition}

\begin{proof}
For the first statement, recall from (the dual verion of) \cite[Proposition 4.3.1]{Borceux2} that for any coalgebra in any monoidal category, the forgetful functor $\cF$ creates colimits. Being $\cC$ abelian, this implies that $\com{H}$ is finitely cocomplete. Using the fact that finite coproducts in $\cC$ are biproducts, it is then a standard verification that the canonical projections are colinear and hence finite coproducts $\com{H}$ are biproducts as well. 

Suppose then that $\cF$ creates equalizers. Since $\cC$ is abelian, and so it has all equalizers, it follows that $\com{H}$ has all equalizers. Combined with the above, this tells us that $\com{H}$ is finitely complete and cocomplete, it has biproducts and that $\cF$ preserves all of these. In particular, kernels and cokernels exist in $\com{H}$ and they can be computed in $\cC$. As $\cC$ is abelian, $\com{H}$ will be abelian as well.

Regarding the equivalent statements, it follows directly from (the dual of) \cite[Proposition 4.3.2]{Borceux2} that \ref{item:abelian4} implies \ref{item:abelian3}. By definition, \ref{item:abelian3} implies \ref{item:abelian1}. To see that \ref{item:abelian1} implies \ref{item:abelian4}, consider any equalizer in $\cC$. Since the functor $-\ot H:\cC \to \com{H}$ is a right adjoint to $\cF$, it preserves equalizers. By assumption \ref{item:abelian1}, the resulting equalizer in $\com{H}$ is preserved by the forgetful functor and hence $-\ot H:\cC\to\cC$ preserves equalizers. 

Finally, if we assume \ref{item:abelian1}, then we already know that $\com{H}$ is abelian and so any monomorphism is an equalizer. Thus, \ref{item:abelian1} implies \ref{item:abelian2}. If instead we assume \ref{item:abelian2}, then we can deduce as above that $- \ot H: \cC \to \cC$ preserves monomorphisms: if $f$ is a monomorphism in $\cC$, then $f$ is a kernel in $\cC$ and so $f \ot H$ is a kernel (whence, a monomorphism) in $\com{H}$. Therefore, $f \ot H$ is a monomorphism in $\cC$. Now, it is a well-known fact that a functor between abelian categories is exact (i.e.~it is left and right exact) if and only if it is right exact and it preserves monomorphisms. 
Hence, under the additional hypothesis that $- \ot H$ preserves coequalizers, \ref{item:abelian2} implies \ref{item:abelian4}.
\end{proof}

Let us remark that it is possible that $\com{H}$ has all equalizers even when $H$ is not left flat. In \cite{Porst} it is indeed shown that when $\cC$ is of the form $\Rmod{k}$ for a commutative ring $k$ or $\Bimod{A}$ for a (possibly non-commutative) ring $A$, $\com{H}$ is always complete. Moreover, in \cite[Example 1.1]{LaiachiPepeLobillo} an example of a non-flat coring is presented whose category of comodules is Abelian (even Grothendieck).

\begin{theorem}\label{thm:uffa}
Let $\cC$ be an abelian monoidal category and let $H$ be a coalgebra in $\cC$ which is left flat. 
Then $\gparcom{H}_{gl} = \gparcom{H}$.
\end{theorem}

\begin{proof}
By proposition \ref{Hflat},  the equalizer \eqref{eq:glob} exists in $\com{H}$ and its object part $Y_X$ can be computed in $\cC$. Hence $Y_X$ is canonically isomorphic to the object $T$ from Proposition \ref{prop:superthm}. Therefore, the result follows immediately from Corollary \ref{cor:globAb}.
\end{proof}

Our next result is an improvement of Theorem \ref{globcov} in the abelian case. 

\begin{theorem}\label{thm:abEquiv}
For a coalgebra $(H,\Delta,\varepsilon)$ in an abelian monoidal category $\cC$ such that $\com{H}$ admits the equalizer \eqref{eq:glob} for every partial comodule $X$ and $\com{H} \to \cC$ preserves them, every proper global cover is minimal. Consequently, the equivalence of Theorem \ref{globcov} induces an equivalence between the categories $\gparcom{H}$ and $\Cov^{H}_{pr}$.
\end{theorem}

\begin{proof}
First of all, Corollary \ref{cor:globAb} 
entails that the full subcategory of globalizable geometric partial comodules $\gparcom{H}_{gl}$ coincides with the ambient category $\gparcom{H}$. 
Moreover, if we consider a proper global covering $(Y,X,p)$ and we perform the pushout
\begin{equation}\label{eq:pushglob}
\begin{gathered}
\newdir{ >}{{}*!/-10pt/@{>}}
\xymatrix @!0 @R=25pt @C=45pt{
 & Y \ar@{->>}[dl]_-{p} \ar@{ >->}[dr]^(0.55){\,\ (p \ot H)\circ \delta_Y} & \\
X \ar@{ >->}[dr]_-{\rho_X} & & X\otimes H \ar@{->>}[dl]^(0.4){\pi_X} \\
 & X\bullet H \ar@{}[u]|<<<{\pushout}
}
\end{gathered}
\end{equation}
then, in view of the fact that $(p \ot H)\circ \delta_Y$ is a monomorphism 
and of Lemma \ref{lemma:A1}, \eqref{eq:pushglob} is also a pullback diagram. By Proposition \ref{prop:superthm},  
this implies that the canonical morphism 
$(\eta_Y,\id_X):(Y,X,p) \to (Y_X,X,\epsilon_X)$ is in fact an isomorphism 
and hence every proper cover is also minimal, as $Y_X$ is minimal.
\end{proof}

\begin{remark}
Notice that, despite $\gparcom{H}_{gl} = \gparcom{H}$ and $\Cov^{H}_{pr,min} = \Cov^{H}_{pr}$, over an abelian category $\cC$ as in Theorem \ref{thm:abEquiv}, it is not necessarily true that every global covering is proper. If we consider the group-like bialgebra $\C[X]$ and the projection $p : \C[X] \to \C, X^n \mapsto \delta_{0,n}$, then $(\C[X],\C,p)$ is a well-defined non-proper global covering.
\end{remark}

In the framework of classical (co)algebras over a commutative ring, the following restatement of Theorem \ref{thm:abEquiv} explains why the category of geometric partial (co)modules contains more information than the category of ordinary (co)modules.

\begin{corollary}\label{gpcalgcoalg}
Let $k$ be a commutative ring. 
\begin{enumerate}[label=(\arabic*),ref=(\arabic*),leftmargin=1cm]
\item If $A$ is a $k$-algebra, then (left) geometric partial modules over $A$ can be identified with pairs $(M,V)$ where $M$ is a (global) $A$-module and $V$ is a chosen generating $k$-submodule of $M$. Morphisms of geometric partial modules are $A$-linear morphisms of the corresponding $A$-modules which map generating submodules to generating submodules.
\item\label{kmod:item2} If $C$ is a $k$-coalgebra that is flat as left $k$-module, then (right) geometric partial comodules over $C$ can be identified with pairs $(M,V)$ where $M$ is a (global) $C$-comodule together with a chosen $k$-submodule $V$ satisfying $\delta(M)\cap (V\ot C)=0$. Morphisms of geometric partial comodules are $C$-colinear morphisms between the corresponding $C$-comodules which restrict to the chosen $k$-submodules. \\
Equivalently, they can be described as global $C$-comodules $M$ together with a chosen co-generating quotient $k$-module $N$. In this case, morphisms of geometric partial comodules are $C$-colinear morphisms between the corresponding $C$-comodules which factors through the co-generating quotient $k$-modules.
\end{enumerate}
\end{corollary}

\begin{proof}
It follows directly from Theorem \ref{thm:abEquiv}, by spelling out the co-generating condition (and keeping in mind that in \ref{kmod:item2} the coalgebra $C$ is assumed to be $k$-flat). For the first claim in \ref{kmod:item2} one needs, more precisely, to observe also that a global comodule $M$ is co-generated by a quotient $k$-module $N$ via a surjective $k$-linear map $p:M\to N$ if and only if $V \coloneqq \ker (p)$ is such that $\delta(M)\cap (V\ot C)=0$.
\end{proof}


\subsection{Partial modules over Hopf algebras (and dilations)}\label{ssec:hopf}

Assume that $\K$ is a field and that $(H,\mu,u,\Delta,\varepsilon,S)$ is a Hopf $\K$-algebra. Take $\cC=\op{\vectk}$, the opposite of the category of vector spaces over $\K$. It is an abelian monoidal category for which Theorem \ref{thm:uffa} holds for any coalgebra in $\cC$. Since $(H,\mu,u)$ is a coalgebra therein, $\gparcom{H}_{gl} = \gparcom{H}$ in view of \S\ref{ssec:abelian}. Recall the following definition from \cite[Definition 5.1 and Remark 5.2]{ParRep}.

\begin{definition}\label{def:ParRep}
A (left) partial module over $H$ is a vector space $M$ together with a linear map $\lambda:H\ot M \to M, h \otimes m\mapsto h \cdot m$, such that
\begin{enumerate}[leftmargin=1cm] 
\item\label{item:pm1} $1_H\cdot m = m$,
\item\label{item:pm2} $h \cdot \left(k_{(1)} \cdot \left(S(k_{(2)}) \cdot m\right)\right) = (hk_{(1)}) \cdot \left(S(k_{(2)}) \cdot m\right)$,
\item\label{item:pm3} $k_{(1)} \cdot \left( S(k_{(2)})\cdot \left(h \cdot m\right)\right) = k_{(1)} \cdot \left(S(k_{(2)})h \cdot m\right)$,
\end{enumerate}
for all $m\in M$, $h,k\in H$. If $(M, \lambda)$ and $(M', \lambda')$ are partial $H$-modules, then a morphism of partial $H$-modules is a $\K$-linear map $f:M\to M'$ satisfying $f\left(h\cdot m\right) = h\cdot'f(m)$ for all $h \in H, m \in M$. The category whose objects are the partial $H$-modules and whose morphisms are the ones defined above is the category $\ParMod_H$.
\end{definition}

\begin{remark}
For the sake of precision, \cite[Remark 5.2]{ParRep} involves five axioms instead of just the three above, but in fact the missing two are redundant. See \cite[Lemma 2.11]{ParCorep}.
\end{remark}

\begin{proposition}\label{prop:PmodGPmod}
Any partial $H$-module can be endowed with a geometric partial comodule structure over the coalgebra $H$ in the abelian monoidal category $\cC={{\sf Vect}_\K}^{\rm op}$, with
\begin{equation}\label{eq:ParRepBul}
H \bullet M = \left\{\left. \sum_{i} h_i \otimes m_i \in H \otimes M \ \right| \ \sum_i k\cdot\left(h_i\cdot m_i\right) = \sum_i kh_i\cdot m_i \ \forall\, k\in H \right\}.
\end{equation}
This induces a faithful functor
\[
\cP : \ParMod_H \to \op{\left(\gparcom{H}\right)}.
\]
\end{proposition}

\begin{proof}
Since a partial $H$-module is, in particular, a counital coaction of the coalgebra $H$ in $\op{\vectk}$ and since $\op{\vectk}$ is a cocomplete category for which Lemma \ref{buliscolimit} 
holds, we can apply Theorem~\ref{thm:constructgeometric} 
to conclude that the first claim holds and  Lemma \ref{buliscolimit} 
itself to explicitly describe $H \bul M$ as the limit of the pairs
\[
\xymatrix @C=25pt{
H \ot M \ar[rr]^-{f \ot H \ot M} && H \ot H \ot M \ar@<+0.5ex>[rr]^-{H \ot \lambda} \ar@<-0.5ex>[rr]_-{\mu \ot M} && H \ot M \ar[r]^-{\lambda} & M
}
\]
for all $f \in \Hom{}{}{\K}{}{\K}{H}$. Concerning the second claim, a morphism of partial modules $f:(M,\lambda) \to (M',\lambda')$ induces a map $H\bul f:H\bul M \to H\bul M'$ which makes diagram \eqref{eq:shield} to commute. Therefore, $\cP(f) = (f, H\bul f)$ is a morphism of partial comodules and $\cP$ is a functor. It is clearly faithful. 
\end{proof}

\begin{definition}[{\cite[Definition 4.1]{AlvesBatistaVercruysse}}]\label{def:dilation}
A dilation for a partial module $(M,\lambda)$ is a triple $(N,T,\theta)$ where: \begin{enumerate*} \item[(i)] $N$ is a global $H$-module with action $\delta:H\ot N\to N, h \ot n \mapsto h \triangleright n$, \item[(ii)] $T$ is a linear endomorphism of $N$ satisfying $T^2 = T$ and\end{enumerate*}
\[
T\left(h_{(1)}\triangleright T\left(S(h_{(2)})\triangleright y\right)\right) = h_{(1)}\triangleright T\left(S(h_{(2)})\triangleright T\left(y\right)\right)
\]
for all $h\in H$, $y\in N$, and \begin{enumerate*} \item[(iii)] $\theta:M \to T(N)$ is an isomorphism of vector spaces satisfying\end{enumerate*}
\begin{equation}\label{eq:thetalin}
\theta\left(h\cdot m\right) = T\left(h \triangleright \theta(m)\right).
\end{equation}
A dilation $(N,T,\theta)$ is called proper if $N$ is generated by $T(N)=\theta(M)$ as an $H$-module and it is called minimal when $N$ does not contain any $H$-submodule that is annihilated by $T$. 
\end{definition}
It is useful to consider the map $\varpi:N\to M, \varpi(n)=\theta^{-1}(T(n))$, so that \eqref{eq:thetalin} becomes
\begin{equation}\label{eq:cdot}
h\cdot m = \varpi\left(h \triangleright \theta(m)\right)
\end{equation}
for all $m\in M$, $h\in H$.

In \cite[Theorem 4.3]{AlvesBatistaVercruysse} it has been proven that every partial module admits a proper and minimal dilation, called the \emph{standard dilation}, which is unique up to isomorphism. Moreover, this construction leads to a functor $\cD: \ParMod_H \to \Rmod{H}$. 

Since $\cC = \op{\vectk}$ is an abelian monoidal category for which 
$\com{H}$ is complete and $\com{H} \to \cC$ preserves limits for every coalgebra $H$,
we know from \S\ref{ssec:abelian} that every geometric partial comodule in $\cC$ is globalizable. Let us now show that, for geometric partial comodules coming from partial $H$-modules, this globalization coincides with the standard dilation.

\begin{theorem}\label{thm:dilglob}
For a partial $H$-module $(M,\lambda)$, the standard dilation $(\cl{M},T_\lambda,\varphi)$ is the globalization of the associated geometric partial comodule in $\cC$.
\end{theorem}

\begin{proof}
Every partial $H$-module $(M,\lambda)$ can be realized as a subspace of $\Hom{}{}{}{}{H}{M}$ via the morphism $\jmath: M \to \Hom{}{}{}{}{H}{M}$ given by $\jmath(m)(h) = h\cdot m$, for all $m\in M$, $h \in H$. Observe that $\Hom{}{}{}{}{H}{M}$ is a global $H$-module with $\delta:H \ot \Hom{}{}{}{}{H}{M} \to \Hom{}{}{}{}{H}{M}, h \ot f \mapsto f\mu(\mbox{-} \otimes h)$.
In \cite[proof of Theorem 4.3]{AlvesBatistaVercruysse}, the standard dilation $\left(\cl{M},T_\lambda,\varphi\right)$ of $M$ has been constructed as the $H$-submodule of $\Hom{}{}{}{}{H}{M}$ generated by the image of the map $\jmath$. 
In particular, we have a commutative diagram
\[
\xymatrix @!0 @C=60pt @R=40pt{
 & \Hom{}{}{}{}{H}{M} & \\
 & \cl{M} \ar[u]_-{\subseteq} & \\
M \ar[ur]^-{\varphi} \ar@/^2ex/[uur]^-{\jmath} & & H \ot M \ar[ul]|-{\delta \circ (H \ot \varphi)} \ar@/_2ex/[uul]|-{\delta \circ (H \ot \jmath)} \\
 & H \bul M \ar[ur]_-{\pi_M} \ar[ul]^-{\rho_M} & 
}
\]
which entails that (the opposite of) $\varphi:M \to \cl{M}$ is a morphism of partial $H$-comodules.

Now, since $\cC$ is abelian, we know from \S\ref{ssec:abelian} that the associated geometric partial comodule $(M,H \bul M,\pi_M,\rho_M)$ admits a globalization $(Y,\nu)$ where $\nu:H \ot Y \to Y, h \ot y \mapsto h \diamond y$.  
By construction, the standard dilation $(\cl{M},T_\lambda,\varphi)$ is a global $H$-comodule in $\cC$ together with a morphism $\op{\varphi}: \cl{M} \to M$ in $\cC$ of partial comodules. Thus, there exists a unique $H$-colinear morphism $\op{\sigma}: \cl{M}\to Y$ in $\cC$ such that $\op{\epsilon_M}\, \op{\circ}\, \op{\sigma} = \op{\left(\sigma \circ \epsilon_M\right)} = \op{\varphi}$ in $\cC$, by the universal property of $Y$.
Let us prove that $\sigma$ is an isomorphism. Since $\cl{M}$, as an $H$-module, is generated by $\varphi(M)$, we have that
\[
x = \sum_i h_i\triangleright \varphi(m_i) = \sum_i h_i\triangleright \sigma\epsilon_M(m_i) = \sigma\left(\sum_i h_i\diamond \epsilon_M(m_i)\right)
\] 
for every $x\in \cl{M}$ and so $\sigma$ is surjective. To prove that it is injective as well, recall 
that the canonical epimorphism $\kappa: H\ot M \to Y$ of the coequalizer
\begin{equation}\label{eq:ParRepCoeq}
\xymatrix@C=45pt{
H \ot H\bul M \ar@<+0.5ex>[rr]^-{H \ot \rho_M} \ar@<-0.5ex>[rr]_-{(\mu \ot M)\circ(H \ot \pi_X)} && H \ot M \ar[r]^-{\kappa} & Y 
}
\end{equation}
in $\vectk$ satisfies $\kappa = \nu \circ \left(H \otimes \epsilon_M\right)$ (see Theorem \ref{thm:globalization}), that is $\kappa (h\ot m) = h \diamond \epsilon_M(m)$ for all $h \in H$, $ m\in M$. Hence $\sigma$ additionally satisfies $\sigma \circ \kappa = \delta \circ (H \ot \varphi)$. Thus, if $y\in Y$ is such that $\sigma (y) =0$, then
\begin{equation}\label{eq:triang0}
0 = \sigma \left(y\right) =  \sigma \left(\kappa\left(\sum_{i}h_i \ot m_i\right)\right) = \sum_{i}h_i\triangleright \varphi\left( m_i\right)
\end{equation}
where $\sum_{i}h_i \ot m_i \in H \ot M$ is such that $\kappa\left(\sum_{i}h_i \ot m_i\right) = y$. However, if $\sum_{i}h_i\triangleright \varphi\left( m_i\right) = 0$,
\[
0 = \varpi \left(\sum_{i}h_i\triangleright \varphi\left( m_i\right)\right) \stackrel{\eqref{eq:cdot}}{=} \sum_i h_i \cdot m_i
\]
and hence
\begin{gather*}
\sum_i k\cdot\left(h_i \cdot m_i\right) = k\cdot\left(\sum_i h_i \cdot m_i\right) = 0 \stackrel{\eqref{eq:triang0}}{=} \varpi\left(k\triangleright\left(\sum_{i}h_i\triangleright \varphi\left( m_i\right)\right)\right) \\
= \sum_{i}\varpi\left(kh_i\triangleright \varphi\left( m_i\right)\right) \stackrel{\eqref{eq:cdot}}{=} \sum_i kh_i \cdot m_i,
\end{gather*}
for all $k\in H$, which entails that $\sum_{i}h_i \ot m_i \in H\bul M$ by \eqref{eq:ParRepBul}. As a consequence,
\begin{gather*}
y = \kappa\left(\sum_{i}h_i \ot m_i\right) = \left(\kappa\circ(\mu\ot M) \circ (H \ot \pi_M)\right)\left(\sum_{i}1 \ot h_i \ot m_i\right) \\
 \stackrel{\eqref{eq:ParRepCoeq}}{=} \left(\kappa\circ (H \ot \rho_M)\right)\left(\sum_{i}1 \ot h_i \ot m_i\right) = \kappa\left(1\ot \sum_{i} h_i \cdot m_i\right) = 0
\end{gather*}
and $\sigma$ is injective.
\end{proof}

\begin{corollary}
Denote by $\cJ : \Rmod{H} \to \ParMod_H$ the inclusion functor and by $\cD: \ParMod_H \to \Rmod{H}$ the dilation functor. We have that $\cG\circ\cP \cong \cD : \ParMod_H \to \Rmod{H}$ and $\cP \circ \cJ \cong \cI: \Rmod{H} \to \ParMod_H$, whence the inner and outer triangle in following diagram commute
\[
\xymatrix @=15pt{
 & \Rmod{H}=\op{\left(\com{H}\right)} \ar@<-0.5ex>[dl]_-{\cJ} \ar@<+0.5ex>[dr]^-{\cI} & \\
\ParMod_H \ar@<-0.5ex>[ur]_-{\cD} \ar[rr]_-{\cP} & & \op{\left(\gparcom{H}\right)} \ar@<+0.5ex>[ul]^-{\cG} 
}
\]
where the globalization functor $\cG$ is now left adjoint of the inclusion functor $\cI$ and where $\gparcom{H}$ and $\com{H}$ denote the categories of geometric partial comodules and of global comodules over the coalgebra $(H,\mu,u)$ in $\cC = \op{\vectk}$, respectively.
\end{corollary}

Remark that the functors $\cD$ and $\cJ$ above are not adjoints. Rather, we have for all $M\in\ParMod_H$ and $N\in \Rmod{H}$, natural transformations 
$$\Rmod{H}\big(\cD (M), N\big)\cong \gparcom{H} \big(\cP (M),\cP\cJ (N)\big) \supset \ParMod_H \big(M,\cJ (N)\big)$$
where the latter inclusion is not an equality (as $\cP$ is faithful but not full).

Let us conclude this subsection by inflecting Proposition \ref{prop:PmodGPmod} in a concrete example of interest and by computing the resulting globalization.

\begin{example}[inspired by {\cite[\S6.2]{ParRep}}]
Let $C_2 = \{1,g\}$ be the cyclic group of order $2$ generated by some element $g$ and let $H \coloneqq \K C_2$ be its group Hopf algebra. By \cite[Theorem 4.2]{ParRep}, partial modules over $\K C_2$ can be identified with modules over the associated partial Hopf algebra $H_{par}$. By \cite[Definition 4.1]{ParRep} (and in view of \cite[Lemma 2.11]{ParCorep}), 
\[H_{par} = T(H)\Big/\left\langle \begin{gathered} 1_{T(H)} - 1_H, \quad h \ot k_{(1)} \ot S(k_{(2)}) - h k_{(1)} \ot S(k_{(2)}), \\ k_{(1)} \ot S(k_{(2)}) \ot h - k_{(1)} \ot S(k_{(2)})h \end{gathered} ~\Bigg|~ h,k \in H\right\rangle,\]
that is to say, by setting $x \coloneqq [g]$ and $y \coloneqq [1]$,
\[H_{par} ~ = ~ \K\langle x,y\rangle\Big/\left\langle \begin{gathered} 1-y, \ y^3 - y^2, \ yx^2 - x^2,\\ xy^2-xy, \ x^3-xy \end{gathered}\right\rangle ~ \cong ~ \K[X]/\langle X^3-X\rangle.\]
Therefore, any partial module over $\K C_2$ is the same thing as a vector space $V$ together with a distinguished linear endomorphism $T$ such that $(T+\id_V)T(T - \id_V) = 0$. The partial $H$-action is uniquely determined by $1\cdot v = v$ and $g \cdot v = T(v)$ for all $v \in V$. 
Decompose $V = V_{-1} \oplus V_0 \oplus V_1$ into its eigenspaces, where $V_0 = \ker(T)$, and set $c(V) \coloneqq V_{-1} \oplus V_1$ (this is the \emph{global core} of $V$ in the sense of \cite[Theorem 2.5]{AlvesBatistaVercruysse}). In this setting,
\[H \bul V = \big\{1 \ot v + g \ot w \mid \forall\,h \in H, h \cdot (g \cdot v) = hg \cdot v\big\} = \big(1 \ot V\big) \oplus \big(g \ot c(V)\big).\]
Let us compute the globalization of $(V,H\bul V,\pi_V,\rho_V)$. To this aim, fix bases $\cB_i$ of $V_i$ for $i \in \{0,\pm1\}$ and notice that
\[H \ot H\bul V = \ms{span}_\K\left\{\begin{gathered} 1 \ot 1 \ot v, \ g \ot 1 \ot v, \\ 1 \ot g \ot w,\ g \ot g \ot w\end{gathered} ~\Big|~ \begin{gathered} v \in \cB_{-1} \cup \cB_0 \cup \cB_1 \\ w \in \cB_{-1} \cup \cB_1 \end{gathered}\right\}.\]
Since the first row is made of elements in the equalizer of $H \ot \rho_V$ and $(\mu \ot V)\circ (H \ot \pi_V)$, we conclude that
\begin{align*}
Y & \coloneqq (H \ot V) \Big/ \ms{span}_\K\left\{\begin{gathered}1 \ot g \cdot v - g \ot v,\ 1 \ot v - g \ot g\cdot v \\ 1 \ot g\cdot w - g \ot w,\ 1 \ot w - g \ot g\cdot w\end{gathered} ~\Big|~ \begin{gathered} v \in \cB_{-1} \\ w \in \cB_1 \end{gathered}\right\} \\
& = (H \ot V) \big/ \ms{span}_\K\left\{ 1 \ot v + g \ot v, 1 \ot w - g \ot w \mid v \in \cB_{-1}, w \in \cB_1\right\}.
\end{align*}
Set $t \coloneqq (1+g)/2$ and $s \coloneqq (1-g)/2$ in $\K C_2$. Then
\[Y \cong (H \ot V) \Big/ \ms{span}_\K\left\{ \begin{gathered} t \ot v, \\ s \ot w \end{gathered} ~\Big|~ \begin{gathered} v \in \cB_{-1}, \\ w \in \cB_1 \end{gathered} \right\} \cong (H \ot V_0) \oplus (s \ot V_{-1}) \oplus (t \ot V_1)\]
with regular left $H$-action and, by writing $v = v_{-1} + v_0 + v_1$,
\[\epsilon: V \to Y, \qquad v \mapsto (1 \ot v_0) + (s \ot v_{-1}) + (t \ot v_1).\]
Notice also that, since $s+t = 1$, $Y \cong (V_0 \oplus V_{-1}) \oplus (V_0 \oplus V_1)$.
This is in accordance (up to isomorphism) with \cite[Example 4.11]{AlvesBatistaVercruysse}, as expected.
\end{example}

\subsection{Partial representations of finite groups}\label{ssec:groups}

Partial representations of groups can be viewed as a particular instance of partial representations of Hopf algebras, and hence the results of the previous section can be applied to recover the dilation of partial group representations from \cite{Abadie2} by means of our globalization. However, in \cite[\S1.2]{MicheleWilliam} an alternative, apparently \emph{ad hoc}, globalization theorem for partial representations of finite groups on complex vector spaces is presented. Here, we show that the globalization studied therein is a particular instance of the globalization for partial modules over Hopf algebras.

\begin{definition}[{\cite[Definition 6.1]{DokuExel}}] \label{def:partial_rep}
A \emph{partial representation} $(V,\pi)$ of a group $G$ on a complex vector space $V$ is a map $\pi\colon G\to \End{\C}{V}$ such that for all $g, h\in G$
\begin{enumerate}[label=(PR\arabic*),ref=(PR\arabic*)]
	\item\label{item:PR1} $\pi(1_G)=\id_V$;
	\item\label{item:PR2} $\pi(g^{-1})\pi(gh)=\pi(g^{-1})\pi(g)\pi(h)$;
	\item\label{item:PR3} $\pi(gh)\pi(h^{-1})=\pi(g)\pi(h)\pi(h^{-1})$.
\end{enumerate}
As a matter of notation, one sets $V_{g} \coloneqq \pi(g)\pi(g^{-1})(V)$ for all $g \in G$. For $(V,\pi)$ and $(V',\pi')$ two partial representations of $G$, a \emph{morphism of partial representations} is a linear map $f\colon V\to V'$ such that $f\circ \pi(g)=\pi'(g)\circ f$ for all $g\in G$. In particular, partial representations of a group form a category that we denote by $\ParRep_G$.
\end{definition}

Any partial representation $(V,\pi)$ of a group $G$ gives rise to a partial module over the group Hopf algebra $\C[G]$ in the sense of Definition \ref{def:ParRep}. Therefore, in view of Proposition \ref{prop:PmodGPmod}, it gives rise to a geometric partial comodule in $\mathsf{Vect}_{\C}^{\sf{op}}$ and hence one may consider the globalization of $V$ in the sense of Definition \ref{def:glob}, which by the foregoing coincides (up to isomorphism) with the proper minimal dilation of $V$ in the sense of \cite[Definition 4.1]{AlvesBatistaVercruysse}. 

Concretely, notice that the globalization $Y_V$ of $(V,\pi)$ can be realized as the quotient $(\C[G] \otimes V)/S$ where $S \subseteq \C[G] \otimes V$ is the vector subspace generated by 
\[
\left\{\left.\sum_{g \in G} hg \otimes v_g - \sum_{g \in G} h \otimes \pi(g)(v_g) \,\right| h \in G \text{ and } \sum_{g\in G}\pi(k)\pi(g)(v_g) = \sum_{g\in G} \pi(kg)(v_g), \forall\, k \in G\right\}.
\]
On the other hand, the globalization $\overline{V}$ from \cite[Theorem 1.18]{MicheleWilliam} is realized as the quotient of $\C[G] \otimes V$ by $Z$, the $\C$-vector subspace generated by
\[
\Big\{hg\otimes v-h\otimes \pi(g)(v) ~\big|~ g,h\in G, v\in V_{g^{-1}}\Big\}.
\]
Once recalled that $\C[G] \otimes V$ has a natural left $\C[G]$-module structure given by multiplication on the left tensorand, one observes that
\[
\begin{gathered}
S = \C[G]\cdot \left\{\left.\sum_{g \in G} g \otimes v_g - \sum_{g \in G} 1 \otimes \pi(g)(v_g) ~\right|~  \sum_{g\in G}\pi(k)\pi(g)(v_g) = \sum_{g\in G} \pi(kg)(v_g), \forall\, k \in G\right\} \\
\text{and } \qquad Z = \C[G]\cdot \Big\{g\otimes v-1\otimes \pi(g)(v) ~\big|~ g\in G, v\in V_{g^{-1}}\Big\}.
\end{gathered}
\]

\begin{proposition}\label{prop:S=Z}
In $\C[G] \otimes V$ we have $S = Z$.
\end{proposition}

\begin{proof}
For every $g,h \in G$ and for all $v \in V_{g^{-1}} = \pi(g^{-1})\pi(g)(V)$ we have that
\[
\pi(h)\pi(g)(v) = \pi(h)\pi(g)\pi(g^{-1})\pi(g)(v) \stackrel{\ref{item:PR3}}{=} \pi(hg)\pi(g^{-1})\pi(g)(v) = \pi(hg)(v),
\]
whence $g\otimes v-1\otimes \pi(g)(v) \in S$ and so $Z \subseteq S$. Conversely, let $\sum_{g \in G} g \otimes v_g \in \C[G] \otimes V$ be such that
\begin{equation}\label{eq:S}
\sum_{g \in G}\pi(k)\pi(g)(v_g) = \sum_{g \in G} \pi(kg)(v_g) \quad \text{ for all } k \in G.
\end{equation}
Observe that
\begin{multline*}
\sum_{g \in G} g \otimes v_g - \sum_{g \in G} 1 \otimes \pi(g)(v_g) + Z \\
= \left(\sum_{g \in G} g \otimes \left(v_g - \pi(g^{-1})\pi(g)(v_g)\right)\right) + \left(\sum_{g \in G} g \otimes \pi(g^{-1})\pi(g)(v_g) - \sum_{g \in G} 1 \otimes \pi(g)(v_g)\right) + Z \\
\stackrel{\ref{item:PR2}}{=} \sum_{g \in G} g \otimes \left(v_g - \pi(g^{-1})\pi(g)(v_g)\right) + Z
\end{multline*}
in $(\C[G] \otimes V)/Z$ and set $w_g \coloneqq v_g - \pi(g^{-1})\pi(g)(v_g)$ for all $g \in G$.
Condition \ref{item:PR3} and \eqref{eq:S} entail that $\sum_{g \in G} \pi(kg)(w_g) = 0$ for all $k \in G$. In particular, if there exists a $h \in G$ such that $v_{h} \neq \pi\left(h^{-1}\right)\pi\left(h\right)\left(v_{h}\right)$, we may take $k = h^{-1}$ and find
\begin{equation}\label{eq:S2}
w_{h} = - \sum_{g \neq h} \pi\left(h^{-1}g\right)\left(w_g\right).
\end{equation}
As a consequence,
\[
\begin{aligned}
\sum_{g \in G} g \otimes w_g & 
= h \cdot \left(\sum_{g \neq h} h^{-1}g \otimes w_g - \sum_{g \neq h} 1 \otimes \pi(h^{-1}g)(w_g)\right)
\end{aligned}
\]
and hence
\[
\sum_{g \in G} g \otimes v_g - \sum_{g \in G} 1 \otimes \pi(g)(v_g) + Z = h \cdot \left(\sum_{g \neq h} h^{-1}g \otimes w_g - \sum_{g \neq h} 1 \otimes \pi(h^{-1}g)(w_g)\right) + Z.
\]
Now, since for all $k \in G$ we have
\[
\sum_{g \neq h} \pi\big(kh^{-1}g\big)(w_g) = \sum_{g \in G} \pi\big(kh^{-1}g\big)(w_g) - \pi(k)(w_h) = - \pi(k)(w_h) \stackrel{\eqref{eq:S2}}{=} \sum_{g \neq h} \pi(k)\pi\big(h^{-1}g\big)(w_g),
\]
and inductive argument leads to conclude that $\sum_{g \in G} g \otimes v_g - \sum_{g \in G} 1 \otimes \pi(g)(v_g) + Z = Z$ and hence that $S \subseteq Z$, thus finishing the proof.
\end{proof}

\begin{corollary}
The globalization of a partial representation $(V,\pi)$ of $G$ as in \cite[Theorem 1.18]{MicheleWilliam} coincides with the globalization from Theorem \ref{thm:globalization}.
\end{corollary}


\subsection{Algebraic partial comodules over Hopf algebras}\label{ssec:pcom}

Let $\K$ be a field and consider a Hopf algebra $(H,\mu,u,\Delta,\varepsilon,S)$ over $\K$, as in \S \ref{ssec:hopf}. Let $\cC = \vectk$, which is again an abelian monoidal category for which Theorem \ref{thm:uffa} holds for any coalgebra, and consider the coalgebra $(H,\Delta,\varepsilon)$ therein. As in \S\ref{ssec:hopf}, we have that $\gparcom{H}_{gl}=\gparcom{H}$.
Recall the following definition from \cite{ParCorep}, which is the categorical dual of Definition \ref{def:ParRep}.

\begin{definition}[{\cite[Definition 3.1, Lemma 3.3]{ParCorep}}]
An \emph{algebraic partial (right) $H$-comodule} is a $\K$-vector space $M$ with a $\K$-linear map
$
\partial_M  : M \to M \otimes H, m \mapsto m_{[0]} \otimes m_{[1]},
$ 
satisfying
\begin{enumerate}[leftmargin=1cm]
\item\label{item:APC1} $(M \otimes \varepsilon)\circ \partial_M  = \id_M$,
\item\label{item:APC2} $(M \otimes H \otimes \mu)\circ (M \otimes H \otimes S \otimes H)\circ (M \otimes \Delta \otimes H)\circ (\partial_M  \otimes H)\circ \partial_M  = (M \otimes H \otimes \mu)\circ (M \otimes H \otimes S \otimes H)\circ (\partial_M  \otimes H \otimes H)\circ (\partial_M  \otimes H)\circ \partial_M $,
\item\label{item:APC3} $(M \otimes \mu \otimes H)\circ (M \otimes S \otimes H \otimes H)\circ (M \otimes H \otimes \Delta)\circ (\partial_M  \otimes H)\circ \partial_M  = (M \otimes \mu \otimes H)\circ (M \otimes S \otimes H \otimes H)\circ (\partial_M  \otimes H \otimes H)\circ (\partial_M  \otimes H)\circ \partial_M $.
\end{enumerate}
A morphism between two algebraic partial $H$-comodules $(M, \partial_M )$ and $(N, \partial_N)$ is a linear map $f : M \to N$ satisfying $\partial_N \circ f = (f \otimes H) \circ \partial_M $. 
\end{definition}

Let us see that any algebraic partial $H$-comodule can be endowed with a geometric partial $H$-comodule structure as well. To this aim, fix a basis $\{h_i \mid i\in I\}$ for the Hopf algebra $H$ and write, for each $i\in I$,
$$\Delta(h_i)=\sum_{j,k} a^i_{j,k} h_j\ot h_k,$$
where only a finite number of the coefficients $a^i_{j,k}\in \K$ is non-zero. Let $(M,\partial_M )$ be an algebraic partial comodule and consider a basis $\{m_\ell \mid \ell\in L\}$ for $M$. Then we can write
$$\partial_M (m_\ell)=\sum_{p,i} b^\ell_{p,i}m_p\ot h_i$$
for each $\ell\in L$, where only a finite number of the coefficients $b^\ell_{p,i}\in \K$ is non-zero. 

\begin{proposition}\label{prop:algparcom}
Let $(M,\partial_M )$ be an algebraic partial comodule over the Hopf algebra $H$. By keeping the notation introduced above, $M$ can be endowed with a structure of geometric partial $H$-comodule $(M,M\bul H,\pi_M,\rho_M)$ by taking $M\bul H \coloneqq M\ot H/Q$, where $Q\subset M\ot H$ is the subspace generated by the elements
$$\sum_{q,p,k} b^\ell_{p,t}b^p_{q,k} m_q\ot h_k - \sum_{q,p,k} b^\ell_{q,p}a^p_{k,t} m_q\ot h_k $$
for all $\ell \in L$, $t \in I$, with $\pi_M: M\ot H\to M\bul H$ the canonical projection and $\rho_M \coloneqq \pi_M\circ \partial_M$.
\end{proposition}

\begin{proof}
This is a straightforward application of Theorem \ref{thm:constructgeometric}. 
Indeed, for $\cC = \vectk$ the conditions of Lemma \ref{buliscolimit} 
are fulfilled and the prescribed colimit $M\bul H$ can be computed as in the statement, since for any $\ell\in L$ and $f \in \Hom{}{}{\K}{}{H}{\K}$ we have
$$(M \ot H \ot f)\big((\partial_M \ot H)\big(\partial_M (m_\ell)\big)\big) = \sum_{q,p,k,i} f(h_i)b^\ell_{p,i}b^p_{q,k} m_q\ot h_k$$
and 
\[(M \ot H \ot f)\big((M\ot \Delta)\big(\partial_M (m_\ell)\big)\big) = \sum_{q,p,k,i} f(h_i)b^\ell_{q,p}a^p_{k,i} m_q\ot h_k. \qedhere\]
\end{proof}

\begin{theorem}\label{thm:pcorep}
Algebraic partial $H$-comodules admit a globalization.
\end{theorem}

\begin{proof}
Since $\cC = \vectk$ is an abelian monoidal category for which $\com{H}$ is complete and $\com{H} \to \cC$ preserves limits for every coalgebra $H$, the result follows from Theorem \ref{thm:uffa}. 
\end{proof}

As we did for partial modules in \S\ref{ssec:hopf}, let us inflect Proposition \ref{prop:algparcom} and Theorem \ref{thm:pcorep} in a concrete example of interest.

\begin{example}
Let $H_4$ be Sweedler's four dimensional Hopf algebra. As an algebra, $H_4$ is generated by two elements $g,y$ subject to the relations $g^2 = 1$, $y^2 = 0$ and $yg = - gy$.
Therefore, as a vector space over $\K$, $H_4$ is generated by $1,g,y,gy$. Its Hopf algebra structure is uniquely determined by $\Delta(g) = g \ot g$ and $\Delta(y) = y \ot g + 1 \ot y$.
Consider $M \coloneqq \K[z]$, the $\K$-vector space of polynomials in $z$, and define
\[\partial:\K[z] \to \K[z] \ot H_4, \qquad z^n \mapsto z^{n+1} \ot y + z^n\ot \frac{1+g}{2}.\]
This makes of $(\K[z],\partial)$ a right algebraic partial $H_4$-comodule (see \cite[Example 3.14]{ParCorep}). To determine $\K[z] \bul H_4$, we compute $M \ot H \ot f$ of $(\partial \ot H)(\partial(z^n)) - (\K[z] \ot \Delta)(\partial(z^n))$ for all $n \geq 0$ and all $f \in \{1^*,g^*,y^*,gy^*\}$, the dual basis of $H_4$. Up to a non-zero scalar, it results in the family of elements
\[U \coloneqq \left\{z^{n+1}\ot y - z^n \ot \frac{1-g}{2} ~\Big|~ n \geq 0\right\}\subset \K[x]\ot H_4.\]
The linear map
\[\rho:\K[z] \to \K[z] \bul H_4 = \frac{\K[z] \ot H_4}{\ms{span}_\K U}, \qquad z^n \mapsto \overline{z^n\ot g},\]
where by $\overline{(-)}$ we denote the equivalence class in the quotient, together with the canonical projection on the quotient $\pi$, equip $\K[z]$ with a structure of geometric partial comodule over $H_4$. To compute the globalization of $(\K[z],\K[z] \bul H_4,\pi,\rho)$ we have to determine the equalizer $Y$ of \eqref{eq:glob}. However, after observing that a basis for $\K[z]\bul H_4 \ot H_4$ is given by
\[\left\{\cl{1 \ot y}\ot a,\cl{z^n \ot b}\ot a ~\big|~ a \in \{1,g,y,gy\}, b\in\{1,g,gy\},n\geq 0\right\},\]
one can conclude by a direct computation that $Y = \ms{span}_\K\{z^n \ot g\mid n \geq0\} = \K[z]\ot g$ with $\delta(z^n \ot g) = (z^n \ot g) \ot g$ and $\epsilon: Y \to \K[z], z^n \ot g \mapsto z^n$. We point out that any element $p(z) \in \K[z]$ determines a one-dimensional geometric partial subcomodule of $\K[z]$ (in accordance with \cite[Theorem 2.27]{JoostJiawey}). However, any positive power of $z$, for example, generates an infinite-dimensional algebraic partial subcomodule (as observed in \cite[Example 3.14]{ParCorep}). This nicely highlights the advantages of working with geometric partial comodules instead of algebraic ones.
\end{example}

\section{Hopf partial comodules over bialgebras}\label{sec:Hopf}

Let $(B,\mu,u,\Delta,\varepsilon)$ be a bialgebra over a field $\K$. The category $\cC \coloneqq \Rmod{B}$ of right $B$-modules is an abelian monoidal category with tensor product $\otimes = \otimes_\K$ and unit object $\I = \K$. Furthermore, $(B,\Delta,\varepsilon)$ with the regular right $B$-module structure is a coalgebra in $\cC$ and the endofunctor $- \otimes B$ is continuous. In particular, $\com{B}$ is complete and $\com{B} \to \cC$ preserves limits.

\begin{lemma}\label{Hopfmodule1}
There is a bijective correspondence between:
\begin{enumerate}[label=(\arabic*),ref=(\arabic*),leftmargin=0.7cm]
\item\label{item:HM1} geometric partial comodules $(M,M\bul B,\rho_M,\pi_M)$ over $(B,\Delta,\varepsilon)$ in $\cC = \Rmod{B}$;
\item\label{item:HM2} right $B$-modules $M$ with a geometric partial comodule structure $(M,M\bul B,\pi_M,\rho_M)$ over $(B,\Delta,\varepsilon)$ in $\vectk$ such that $M\bul B$ is a right $B$-module and $\rho_M$ and $\pi_M$ are $B$-linear.
\end{enumerate}
\end{lemma}

\begin{proof}
The conditions of \ref{item:HM2} assure that $(M,M\bul B,\pi_M,\rho_M)$ is a partial comodule datum in $\Rmod{B}$. Moreover, since
the forgetful functor $\Rmod{B}\to \vectk$ preserves and reflects pushouts, the counitality and geometric coassociativity conditions in $\Rmod{B}$ or $\vectk$ are the same. 
\end{proof}

This leads to the following definition.

\begin{definition}
A (right) \emph{Hopf partial comodule} is a quadruple $(M,M\bul B,\pi_M,\rho_M)$ satisfying the equivalent conditions of Lemma~\ref{Hopfmodule1}. A morphism of Hopf partial comodules is a $B$-linear map that is at the same time a morphism of geometric partial comodules. We denote the category of Hopf partial comodules by $\gphopfmod{B}$.

Similarly one can define {\em left} Hopf partial comodules and dually {\em Hopf partial modules}, which have a global $B$-comodule structure and a partial $B$-module structure.
\end{definition}

\begin{example}
\begin{enumerate}[leftmargin=0.7cm]

\item Every right Hopf module over $B$ is a right Hopf partial comodule (these are the global $B$-comodules in $\Rmod{B}$).

\item Every right $B$-module is a right Hopf partial comodule over $B$ with the trivial partial comodule structure (see \cite[\S2.4]{JoostJiawey})
\[
\begin{gathered}
\xymatrix@R=7pt{
M \ar@{=}[dr] & & M\otimes B \ar@{->>}[dl]^(0.45){M \otimes \varepsilon} \\
 & M. & 
}
\end{gathered}
\]
\item Let $V$ be a vector space and let $N$ be a right $B$-submodule of $V \ot B$.
Then the geometric partial $B$-comodule induced by the global comodule $(V\ot B,V\ot \Delta)$ on the quotient $(V\ot B)/N$ makes of it a Hopf partial comodule.
\end{enumerate}
\end{example}

Global right Hopf modules over a bialgebra $B$ can also be characterized as modules in the monoidal category of right $B$-comodules. Even if the category of geometric partial comodules over $B$ is not monoidal, in general, (see \cite[\S3.2]{JoostJiawey} and, in particular, \cite[Theorem 3.9]{JoostJiawey}), we may still obtain a similar description in the partial case. 
Recall from \cite[Proposition 3.4]{JoostJiawey} that 
we can define an associative tensor product $\boxtimes$ of partial comodule data as follows. For $(X,X\bul B,\pi_X,\rho_X)$ and $(Y,Y\bul B,\pi_Y,\rho_Y)$ two partial comodule data, their tensor product is given by the outer cospan in the following diagram
\[
\xymatrix @!0 @C=75pt @R=30pt{
X\ot Y \ar[dr]_(0.4){\rho_X\ot \rho_Y} & & X\ot B\ot Y\ot B \ar[dl]_(0.55){\pi_X\ot \pi_Y} \ar[dr]^(0.55){\mu_{X,Y}} & & X \ot Y \ot B \ar@{=}[dl] \\
& (X\bul B)\ot(Y\bul B) \ar[dr]_-{\bar\mu} & & X\ot Y\ot B \ar[dl]^-{\pi_{X\ot Y}} & \\
& & (X\ot Y)\bul B \ar@{}[u]|<<<{\pushout}  & &
}
\]
where $\mu_{X,Y} \coloneqq (X\ot Y\ot \mu)\circ (X\ot\tau_{B,Y}\ot B)$ and $\tau_{B,Y}(b \ot y) = y \ot b$ for all $y \in Y, b \in B$. This (see \cite[Proposition 3.4]{JoostJiawey}) makes of the category $\pcd{B}$ of partial comodule data over $B$ a monoidal category with oplax unit (i.e., the unitors $\clambda_X: X \to \I \ot X$ and $\crho_X : X \to X \ot \I$ are just natural transformations and not natural isomorphisms) in such a way that
\[
\xymatrix @=10pt{
\com{B} \ar[rr] \ar[dr] & & \pcd{B} \ar[dl] \\
 & \vectk & 
}
\]
is a commutative diagram of monoidal functors. The oplax unit is $(\K,B,\id_B,u)$. 

Note that, in general, even if $X$ and $Y$ are geometric partial comodules, $X\ot Y$ is not necessarily geometric (see \cite[page 4125]{JoostJiawey}). 

Since the tensor product in $\pcd{B}$ of global $B$-comodules coincide with the structure of (geometric) partial comodule datum on their tensor product as global $B$-comodules (see \cite[Remark 3.10]{JoostJiawey}), the bialgebra $B$ itself becomes a monoid in $(\pcd{B},\boxtimes,\K)$ by considering $u:\K \to B$ and $\mu:B \boxtimes B \to B$, that is to say,
\[
\begin{gathered}
\xymatrix @!0 @R=20pt @C=50pt {
\K \ar[dd]_-{u} \ar[dr]^-{u} & & B \ar[dd]^-{u \ot B} \ar@{=}[dl] \\
 & B \ar[dd]^-{u \ot B} & \\
B \ar[dr]_-{\Delta} & & B \ot B \ar@{=}[dl] \\
 & B \ot B &
}
\end{gathered}
\qquad\text{and}\qquad
\begin{gathered}
\xymatrix @!0 @C=80pt @R=20pt{
B\ot B \ar[dd]_-{\mu} \ar[dr]_(0.4){\delta_{B \ot B}} & & (B\ot B)\ot B \ar@{=}[dl] \ar[dd]^-{\mu\ot B} \\
& (B \ot  B)\ot B \ar[dd]^-{\mu \ot B} & \\
B \ar[dr]_-{\Delta} & & B \ot B \ar@{=}[dl] \\
 & B \ot B &
}
\end{gathered}
\]
where $\delta_{B \ot B}(a \ot b) = a_{(1)} \ot b_{(1)} \ot a_{(2)}b_{(2)}$ for all $a,b \in B$ (in Sweedler's notation).

\begin{proposition}
Let $B$ be a $\K$-bialgebra and let $(M,M\bul B,\rho_M,\pi_M)$ be a geometric partial comodule over $(B,\Delta,\varepsilon)$ in $\vectk$. Then there is a bijective correspondence between:
\begin{enumerate}[label=(\arabic*),ref=(\arabic*),leftmargin=0.7cm]
\item\label{item:parhopfcom1} $B$-module structures on $M$ turning $(M,M\bul B,\pi_M,\rho_M)$ in a Hopf partial comodule with the additional property that $\ker(\pi_M)$ is a $B\ot B$-submodule of $M\ot B$;
\item\label{item:parhopfcom2} Actions of $(B,B\ot B,\id_{B \ot B},\Delta)$ on $(M,M\bul B,\pi_M,\rho_M)$ in $(\pcd{B},\boxtimes,\K)$.
\end{enumerate}
In case $B$ is moreover a Hopf algebra, then the additional property in \ref{item:parhopfcom1} is equivalent to the fact that the partial comodule datum $(M,M\bul B,\pi_M,\rho_M)$ is right equivariant in the sense of \cite[Definition 3.5]{JoostJiawey}, which means that $\ker(\pi_M)$ is a right $B$-submodule of $M\ot B$ with respect to the free $B$-action on the right tensorand.
\end{proposition}

\begin{proof}
Let us begin by assuming that we have an associative and unital $B$-module structure $(\mu_M,\mu_M\bul B): M \boxtimes B \to M$ in $\pcd{B}$ over $\cC = \vectk$ (that is, \ref{item:parhopfcom2}). This means that we have a commutative diagram
\begin{equation}\label{eq:supershield}
\begin{gathered}
\xymatrix @!0 @C=75pt @R=38pt{
M\ot B \ar@{}[ddddrr]|-{(c)} \ar[dd]_-{\mu_M} \ar[dr]_(0.4){\rho_M\ot \Delta} & & M\ot B\ot B\ot B \ar[dl]_(0.55){\pi_M\ot B \ot B} \ar[dr]|(0.5){(M\ot B\ot \mu)\circ (M\ot\tau_{B,B}\ot B)} \ar@{}[dd]|-{(a)} & & M \ot B \ot B \ar@{=}[dl] \ar[dd]^-{\mu_M \ot B} \ar@{}[ddddll]|-{(b)} \\
 & (M\bul B)\ot(B\ot B) \ar[dr]_-{\bar\mu} & & M\ot B\ot B \ar[dl]^-{\pi_{M\ot B}} & \\
M \ar[ddrr]_-{\rho_M}  & & (M\ot B)\bul B \ar@{}[u]|<<<{\pushout}  \ar[dd]_-{\mu_M\bul B}  & & M \ot B \ar[ddll]^-{\pi_M} \\
 & & & & \\
 & & M \bul B & & 
}
\end{gathered}
\end{equation}
where $\mu_M : M \ot B \to M$ is associative and unital in the classical sense, as an action in $\vectk$. 
In this case we may consider the following composition
\[\mu'_{M\bul B} \coloneqq \left(M \bul B \ot B\ot B \xrightarrow{\bar\mu} (M \ot B) \bul B \xrightarrow{\mu_M \bul B} M \bul B\right).\]
The associativity of $\mu_M$ together with the fact that $\pi_M \ot B \ot B \ot B \ot B$ is an epimorphism entail that $(\mu_M \bul B) \circ \bar\mu$ equips $M \bul B$ with a $B \ot B$-module structure. 
Furthermore, the commutativity of the following diagram
\[
\xymatrix{
(M \ot B) \ot (B \ot B) \ar[rrr]^-{\pi_M \ot (B \ot B)} \ar[dr]|-{(M \ot B \ot \mu)\circ (M \ot \tau_{B,B} \ot B)} \ar[dd]_-{\mu'_{M \ot B}} & \ar@{}[dr]|-{(a)} & & (M\bul B) \ot (B \ot B) \ar[dl]^-{\bar\mu} \ar[dd]^-{\mu'_{M \bul B}} \\
 & (M \ot B) \ot B \ar[r]^-{\pi_{M \ot B}} \ar[dl]_-{\mu_M \ot B} & (M \ot B) \bul B \ar[dr]^-{\mu_M \bul B} & \\
M \ot B \ar[rrr]_-{\pi_M} & \ar@{}[ur]|-{(b)} & & M \bul B
}
\]
entails that $\pi_M$ becomes a morphism of $B \ot B$-modules and hence $\ker(\pi_M)$ inherits a $B \ot B$-module structure via $\mu'_{M \bul B}$. To conclude, notice that the commutativity of
\[
\xymatrix @!0 @R=50pt @C=130pt{
M \ot B \ar@{}[ddddr]|-{(c)} \ar[dr]^-{\rho_M \ot B} \ar[ddd]_-{\mu_M} & & M \ot B \ot B \ar[dl]_-{\pi_M \ot B} \ar[d]|-{M \ot B\ot \Delta} \ar@/^2.5cm/[ddd]^-{\mu_{M \ot B}} \\
 & M \bul B \ot B \ar[d]_-{M\bul B \ot \Delta} & M \ot B \ot B \ot B \ar[d]|-{\qquad (M \ot B \ot \mu)\circ (M \ot \tau_{B,B} \ot B)} \ar[dl]_-{\pi_M \ot B \ot B} \ar@{}[ddl]|-{(a)} \\
 & M\bul B \ot B\ot B \ar[d]_-{\bar\mu} & M \ot B \ot B \ar[d]|-{\mu_M \ot B} \ar[dl]_-{\pi_{M \ot B}} \ar@{}[ddl]|-{(b)} \\
M \ar[dr]_-{\rho_M} & (M \ot B) \bul B \ar[d]_-{\mu_M \bul B} & M \ot B \ar[dl]^-{\pi_M} \\
 & M\bul B & 
}
\]
implies that both $\rho_M$ and $\pi_M$ are morphisms of $B$-modules in $\vectk$.

Conversely, assume that $(M,M\bul B,\pi_M,\rho_M)$ is a Hopf partial comodule and that $\ker(\pi_M)$ is a $B \ot B$-submodule of $M \ot B$ with respect to the $B \ot B$-module structure
\[\mu'_{M \ot B} \coloneqq \left(M \ot B \ot B \ot B \xrightarrow{M \ot \tau_{B,B} \ot B} M \ot B \ot B \ot B \xrightarrow{\mu_M \ot \mu} M \ot B \ot B\right).\]
This means that $M \bul B$ inherits a $B \ot B$-module structure $\mu'_{M \bul B}$ induced by $\mu'_{M \ot B}$, that is,
\begin{equation}\label{eq:muprime}
\begin{gathered}
\xymatrix @R=30pt{
M \ot B \ot B \ot B \ar[rr]^-{\pi_M \ot B \ot B} \ar[d]|-{(M \ot B \ot \mu) \circ (M \ot \tau_{B,B} \ot B)} & & M \bul B \ot B \ot B \ar[d]^-{\mu'_{M \bul B}} \\
M \ot B \ot B \ar[r]_-{\mu_M \ot B} & M \ot B \ar[r]_-{\pi_M} & M \bul B
}
\end{gathered}
\end{equation}
commutes and hence there exists a unique $\mu_{M} \bul B : (M \ot B) \bul B \to M \bul B$ such that 
\begin{equation}\label{eq:bulmu}
(\mu_M\bul B) \circ \bar\mu = \mu'_{M \bul B} \qquad \text{and} \qquad (\mu_M\bul B) \circ \pi_{M \ot B} = \pi_M \circ (\mu_M \ot B).
\end{equation}
Notice also that since
\begin{align*}
\mu'_{M \bul B} & \circ (M \bul B \ot\Delta) \circ (\pi_M \ot B) = \mu'_{M \bul B} \circ (\pi_M \ot B \ot B) \circ (M \ot B \ot\Delta) \\
& \stackrel{\eqref{eq:muprime}}{=} \pi_M \circ (\mu_M\ot B) \circ (M \ot B \ot \mu) \circ (M \ot \tau_{B,B} \ot B) \circ (M \ot B \ot\Delta) \\
& \stackrel{\phantom{\eqref{eq:muprime}}}{=} \pi_M \circ \mu_{M \ot B} \stackrel{(*)}{=} \mu_{M \bul B} \circ (\pi_M \ot B),
\end{align*}
where in $(*)$ we used the fact that $\pi_M$ is right $B$-linear by hypothesis, then we have 
\begin{equation}\label{eq:muprimeagain}
\mu_{M \bul B} = \mu'_{M \bul B} \circ (M \bul B \ot\Delta)
\end{equation}
because $\pi_M \ot B$ is an epimorphism. Moreover, since
\[(\mu_B \bul B) \circ \bar\mu \circ (\rho_M \ot \Delta) \stackrel{\eqref{eq:bulmu}}{=} \mu'_{M \bul B} \circ (\rho_M \ot \Delta) \stackrel{\eqref{eq:muprimeagain}}{=} \mu_{M\bul B}\circ (\rho_M \ot B) \stackrel{(*)}{=} \rho_M \circ \mu_M,\]
where in $(*)$ we used the fact that $\rho_M$ is right $B$-linear,
we have that \eqref{eq:supershield} commutes and so $(\mu_M,\mu_M\bul B)$ is an action of $(B, B \ot B, \id_{B \ot B},\mu)$ on $(M,M \bul B,\pi_M,\rho_M)$ in $(\pcd{B},\boxtimes,\K)$. 

It is clear that the correspondence between \ref{item:parhopfcom1} and \ref{item:parhopfcom2} we just constructed is bijective.

Finally, the last claim follows from the fact that if $B$ is a Hopf algebra with antipode $S$, then for every $a,b \in B$ we have $a \ot b = a_{(1)} \ot a_{(2)}S\left(a_{(3)}\right)b$. Clearly, if $\ker(\pi_M)$ is, in addition, a $B \ot B$-submodule of $M \ot B$, then in particular it is a $B$-submodule with respect to the regular right $B$-action on the right tensorand. Conversely, if $\ker(\pi_M)$ is, in addition, a $B$-submodule of $M \ot B$ with respect to multiplication by $B$ on the right tensorand, then
\[\sum_i m_i \cdot x \ot b_iy = \sum_i m_i\cdot x_{(1)} \ot b_ix_{(2)}S\left(x_{(3)}\right)y \in \ker(\pi_M)\]
for all $\sum_i m_i \ot b_i \in \ker(\pi_M)$ and hence it is, in fact, a $B \ot B$-submodule of $M \ot B$.
\end{proof}


\begin{theorem}\label{thm:partialHopf}
Right Hopf partial comodules over $B$ are globalizable in the sense of Definition \ref{def:glob}. That is, ${\gphopfmod{B}}_{,\,gl} = \gphopfmod{B}$.
\end{theorem}

\begin{proof}
It follows directly from Theorem \ref{thm:uffa}.
\end{proof}

By combining the globalization theorem with the structure theorem for Hopf modules, we obtain the following.

\begin{theorem}[Fundamental Theorem for Hopf partial comodules over a Hopf algebra]\label{fundamentalHopf}
Let $H$ be a Hopf algebra over $\K$. 
Then the category $\gphopfmod{H}$ is equivalent with the category whose objects are pairs $(V,N)$ composed by a vector space $V$ and an $H$-submodule $N$ of $V\ot H$ such that $(V\ot \Delta(H))\cap (N\ot H)=0$, and whose morphisms $(V,N)\to (V',N')$ are given by $\K$-linear maps $f:V\to V'$ such that $(f\ot H)(N)\subset N'$.
\end{theorem}

\begin{proof}
By Theorem~\ref{thm:partialHopf}, every Hopf partial comodule $M$ is globalizable to a (global) Hopf module $\cG(M)$. By applying the structure theorem of Hopf modules, we know that $\cG(M)\cong V\ot H$, where $V=\coinv{\cG(M)}{H}$ is the space of coinvariant elements. The description of partial comodules in Corollary~\ref{gpcalgcoalg} tells now that $M$ (as a geometric partial comodule in $\vectk$) is a quotient of $V\ot H$ by a subspace $N$ satisfying $(V\ot \Delta(H))\cap (N\ot H)=0$, with the induced partial comodule structure. Finally, the compatibility between the $H$-module and the partial $H$-comodule structure of $M$ (see Lemma~\ref{Hopfmodule1}) assures that $N$ should be an $H$-submodule of $V\ot H$. 
\end{proof}

\begin{corollary}\label{cor:fieldcase}
Let $H$ be a Hopf algebra with antipode $S$ and $\phi: H \to \K$ be a morphism of algebras. Up to isomorphism, the only Hopf partial $H$-comodule structure on the right $H$-module $\K_\phi$ is the trivial partial comodule structure.
\end{corollary}

\begin{proof}
In a nutshell, the globalization of $\K$ should be a non-zero Hopf submodule of $\K\ot H\cong H$, hence it can only be $H$ and, up to isomorphism, the only induced partial comodule structure from $H$ to $\K$ is the trivial one. 

More precisely, the trivial partial comodule structure on $\K$ makes of it a Hopf partial comodule. Now, assume that $(\K,\K \bul H,\pi_\K,\rho_\K)$ is another Hopf partial comodule structure on $\K_\phi$. We already know that $\K$ is globalizable and its globalization $(\cG(\K),\epsilon_\K)$ is of the form $\cG(\K) \cong \coinv{\cG(\K)}{H} \ot H$. Recall that $\kappa \coloneqq (\epsilon_\K \ot H)\circ \delta : \cG(\K) \to \K \ot H$ should be a monomorphism of Hopf modules over $H$ and hence $\coinv{\kappa}{H} : \coinv{\cG(\K)}{H} \to \coinv{\K \ot H}{H} \cong \K$ should be injective, which implies that $\cG(\K)\cong H$ as Hopf modules over $H$. Thus, we may assume that $\cG(\K)=H$. The canonical morphism $\epsilon_\K : H \to \K$ is uniquely determined by $\epsilon_\K(x) = \epsilon_\K(1_H)\phi(x)$ and, being $\epsilon_\K$ surjective, $\epsilon_\K(1_H) \neq 0$. Therefore, up to rescaling, we may also assume that $\epsilon_\K = \phi$. As a consequence, $\kappa = (\phi \ot H)\circ \Delta$ and we have a pushout square
\[
\xymatrix @!0 @R=20pt @C=50pt{
 & H \ar[dl]_-{\phi} \ar[dr]^(0.55){(\phi \ot H)\circ \Delta} & \\
\K \ar[dr]_-{\rho_\K} & & \K \ot H \ar[dl]^-{\pi_\K} \\
& \K \bul H. \ar@{}[u]|<<<{\pushout} & 
}
\]
Next, we observe that $(\phi \ot H)\circ \Delta :H \to \K \ot H, x \mapsto 1 \ot \phi(x_{(1)})x_{(2)}$, is an isomorphism of right $H$-modules with inverse $\K \ot H \to H, 1 \ot x \mapsto \phi S(x_{(1)})x_{(2)}$. As a consequence, also $(\K,\id_\K,\K \ot \varepsilon)$ is a pushout of $(\phi,(\phi \ot H)\circ \Delta)$ and so there exists a unique isomorphism of right $B$-modules $\sigma : \K \bul H \to \K$ such that $\sigma \circ \rho_\K = \id_\K$ and $\sigma \circ \pi_\K = \K \ot \varepsilon$. A quick observation reveals that $\sigma$ is, in fact, an isomorphism of geometric partial comodules.
\end{proof}

Let us conclude with a couple of concrete examples.

\begin{example}
Let $G$ be a group with neutral element $e$ and let $M$ be a representation of $G$ over $\K$ (that is, a left $\K G$-module). A left $\K G$-Hopf module structure on $M$ is the same as a $G$-grading $\{M_g \mid g\in G\}$ on $M$ such that $h\cdot M_g \subseteq M_{hg}$ for all $g,h \in G$. Inspired by this, we call \emph{partially graded $G$-representations} the left Hopf partial comodules over $\K G$.

By applying Theorem~\ref{fundamentalHopf}, partially graded $G$-representations are all and only of the following form. For a vector space $V$, consider $\K G \ot V$ with action $h \cdot (g \ot v) = hg \ot v$ and coaction $\delta(g \ot v) = g \ot g \ot v$ for all $h,g\in G$, $v\in V$. Now consider a $\K G$-submodule $N \subseteq \K G \otimes V$ that does not contain homogeneous elements (that is, such that $N \cap (g \ot V) = 0$ for all $g \in G$ or, equivalently, 
$(\Delta(\K G)\ot V)\cap (\K G\ot N)=0$) and define $M \coloneqq (\K G \ot V) / N$. Then $M$ is a partially graded $G$-representation with the induced structure
\[
\begin{gathered}
\xymatrix @!0 @R=24pt @C=65pt {
 & \K G \ot V \ar[dr]^(0.55){\quad (\K G\otimes p)\circ(\Delta \ot V)} \ar@{->>}[dl]_-{p} & \\
 M \ar[dr]_-{\rho} & & \K G\otimes M \ar@{->>}[dl]^-{\pi} \\
 & \K G\bul M. \ar@{}[u]|<<<{\pushout} &
}
\end{gathered}
\]
\end{example}

\begin{example}
This example shows that if $B$ is a bialgebra which is not a Hopf algebra, then the conclusion of Corollary \ref{cor:fieldcase} does not hold. Consider the monoid bialgebra over $\N$, that is, $\K[X]$ with $X$ group-like, and consider the algebra morphism $\phi: \K[X] \to \K, X \mapsto 0,$ making of $\K$ a $\K[X]$-module. The ideal $X\K[X]$ generated by $X$ is a $\K[X]$-Hopf module because it is also a subcoalgebra. Therefore, $Y \coloneqq \K \ot X\K[X]$ is a right $\K[X]$-Hopf module with structures uniquely determined by
\[(1 \ot X^n) \cdot X = 0 \qquad \text{and} \qquad \delta(1 \ot X^n) = (1 \ot X^n) \ot X^n\]
for $n\geq 1$. The restriction of $\K \ot \varepsilon$ to $Y$ provides a surjective right $\K[X]$-linear morphism $p:Y \to \K$ whose kernel is $\K \ot N$ where $N$ is the ideal $X(X-1)\K[X]$. It follows that
\[\K \xrightarrow{\rho_\K} \K \ot \frac{\K[X]}{\langle X(X-1) \rangle} \xleftarrow{\pi_\K} \K \ot \K[X]\]
where $\rho_\K(1) = 1 \ot X$, is a non-trivial geometric partial comodule structure on $\K$ in $\Rmod{\K[X]}$.
\end{example}


\addtocontents{toc}{\protect\setcounter{tocdepth}{2}}

\end{document}